\documentclass[12pt]{article}

\usepackage{amsmath,amsfonts,latexsym,amssymb,amsthm}

\usepackage{color}

\newcommand{\nA}{\mathbb{A}}
\newcommand{\nc}{\mathbb{C}}

\newcommand{\np}{\mathbb{P}}

\newcommand{\nr}{\mathbb{R}}
\newcommand{\nz}{\mathbb{Z}}

\newcommand{\B}{\mathbf{B}}
\newcommand{\E}{\mathbf{D}}
\newcommand{\M}{\mathbf{M}}
\newcommand{\X}{\mathbf{X}}

\newcommand{\ad}{\mathrm{ad}}

\newcommand{\Ker}{\mathrm{Ker}}

\newcommand{\id}{\mathrm{id}}

\newcommand{\Res}{\mathrm{Res}}
\newcommand{\Supp}{\mathrm{Supp}\,}

\renewcommand{\mod}{\,\mathrm{mod}\,}

\newtheorem{thm}{Theorem}
\newtheorem{lemma}{Lemma}
\newtheorem{cor}{Corollary}
\newtheorem{ack}{Acknowledgment}

\newcommand{\fg}{\mathfrak{g}}

\newcommand{\fz}{\mathfrak{z}}

\newcommand{\fS}{\mathfrak{S}}
\newcommand{\fZ}{\mathfrak{Z}}
\newcommand{\fsl}{\mathfrak{sl}}

\renewcommand{\div}{\mathrm{div}}

\newcommand{\Hom}{{\mathrm {Hom}}}
\newcommand{\Spec}{\mathrm{Spec}}

\newcommand{\bW}{\mathbf{W}}

\newcommand{\cD}{\mathcal{D}}

\newcommand{\cG}{\mathcal{G}}
\newcommand{\cH}{\mathcal{H}}

\newcommand{\cN}{\mathcal{N}}

\newcommand{\cP}{\mathcal{P}}
\newcommand{\cS}{\mathcal{S}}

\renewcommand{\d}{\mathit{d}}
\newcommand{\D}{\dfrac{d}{dz}}

\font\rom cmr12
\font\small cmr10
\font\smit cmti10

\begin{document}

\title{A Global version of Grozman's theorem}
\author{Kenji IOHARA and Olivier MATHIEU}
\maketitle

\begin{abstract}
Let $X$ be a manifold. 
The classification of all equivariant bilinear
maps between tensor density modules over $X$ has been investigated by
Yu. Grozman \cite {G}, who has provided a full classification for those which are 
{\smit differential} operators. Here we investigate the same question
without the hypothesis that the maps are differential operators. 
In our paper, the geometric context is algebraic geometry and the
manifold $X$ is the circle $\Spec\, \nc[z,z^{-1}]$.

Our main motivation comes from the fact that such a classification
is required to complete the proof of the main result of \cite{IM}. 
Indeed it requires to also include the case of deformations
of tensor density modules. 
\end{abstract}

\setcounter{tocdepth}{1}
\tableofcontents

\setcounter{section}{-1}

\section{Introduction}

The  introduction is organized as follows.
The first section is devoted to the main definitions and the statement 
of  the Grozman Theorem. In the second section,  our result
is stated. In  the last section, the main
ideas of the proof are explained.

\subsection{Grozman's Theorem:}\label{sect_0.1}
Let $X$ be a manifold of dimension $n$, let $W_X$ be the Lie algebra of
vector fields over $X$ and let $M,\,N$ and $P$ be three
tensor density modules over $X$. The precise meaning of
tensor density module will be clarified later on and the geometric 
context 
(differential geometry, algebraic geometry, 
$\dots$) is not yet precised.

In a famous paper \cite{G},
Yu. Grozman has classified all bilinear differential
operators $\pi:M\times N\rightarrow P$ which are 
$W_X$-equivariant. Since differential operators are
local \cite{P}, it is enough to consider the case of the
formal geometry, namely
$X=\Spec\, \nc[[z_1,\dots,z_n]]$. 
The most intersting and difficult part of Grozman's theorem involves the
case where $\dim X=1$, indeed the general case follows from 
this case. 

Therefore, we will now assume that  $X=\Spec\, \nc[[z]]$. For this
manifold, the tensor density modules are the modules 
$\Omega^\delta$, where the parameter 
$\delta$ runs over $\nc$. 
As a $\nc[[z]]$-module, $\Omega^\delta$ is 
a rank one free module whose generator is denoted by
$(dz)^\delta$.
The structure
of $W_X$-module on $\Omega^\delta$ is described by the following
formula:

$$\xi.[f.(dz)^\delta]
=(\xi.f+\delta f\div(\xi)).(dz)^\delta$$

\noindent
for any $f\in \nc[[z]]$ and $\xi\in W_X$, where,
as usual, 
$\xi.f=gf'$, $\div( \xi)=g'$ whenever
$\xi=g \D$ for some $g\in \nc[[z]]$. When $\delta$
is a non-negative integer, 
$\Omega^\delta$ is the space $(\Omega^1_X)^{\otimes\delta}$, where 
$\Omega^1_X$ is the space of 
K\"{a}lher differential of $X$. 

The space $\oplus_{\delta}\,\Omega^\delta$  can be realized as the space
of symbols of twisted  pseudo-differential operators on the circle
(see e.g. \cite{IM}, {\it twisted} means that complex powers of 
$\D$ are allowed) and therefore
it carries a structure of 
Poisson algebra. The Poisson structure (a commutative product $P$ and a Lie
bracket $B$) induces   two series of $W_X$-equivariant bilinear maps, namely the maps
$P^{\delta_1,\delta_2}:\Omega^{\delta_1}\times \Omega^{\delta_2}
\rightarrow \Omega^{\delta_1+\delta_2}$ and
the map
$B^{\delta_1,\delta_2}:\Omega^{\delta_1}\times \Omega^{\delta_2}
\rightarrow \Omega^{\delta_1+\delta_2 +1}$. These
operators are explicitly defined by:

$P^{\delta_1,\delta_2}(f_1.(dz)^{\delta_1},f_2.(dz)^{\delta_2})=f_1f_2
(dz)^{\delta_1+\delta_2}$

$B^{\delta_1,\delta_2}(f_1.(dz)^{\delta_1},f_2.(dz)^{\delta_2})=(
\delta_2f_1'f_2-\delta_1f_1f_2') (dz)^{\delta_1+\delta_2+1}$

Moreover, the de Rham operator is a $W_X$-equivariant map
$\d:\Omega^0\rightarrow\Omega^1$. So we can obtained additional
$W_X$-equivariant bilinear  maps between
tensor density module by various compositions of
$B^{\delta_1,\delta_2}$ and $P^{\delta_1,\delta_2}$ with $\d$. 
An example is provided by the map  
$B^{1,\delta}\circ (\d\times id):\Omega^0\times \Omega^\delta
\rightarrow \Omega^{\delta+2}$.
Following Grozman, 
the {\it classical} $W_X$-equivariant bilinear maps are (the linear
combinations of) the maps $B^{\delta_1,\delta_2}, P^{\delta_1,\delta_2}$, and those obtained 
by various compositions with $\d$. 

Grozman discovered one additional $W_X$-equivariant bilinear map,
namely Grozman's operator
$G: \Omega^{-2/3}\times \Omega^{-2/3}
\rightarrow \Omega^{5/3}$ defined by the formula:

$G(f_1.(dz)^{-2/3},f_2.(dz)^{-2/3})=
[2(f_1'''f_2-f_2'''f_1)+3(f_1''f_2'-f_1'f_2'')](dz)^{5/3}.
$

With this, one can state Grozman's result:

\bigskip
{\bf Grozman Theorem.}
{\it Any differential $W_X$ -equivariant bilinear map 
$\pi:\Omega^{\delta_1}\times \Omega^{\delta_2}
\rightarrow \Omega^{\gamma}$ 
between tensor density modules is either classical,
or it is a scalar multiple of the Grozman operator.}

\bigskip
\subsection{The result of the present paper:}\label{sect_0.2}
In this paper, a similar question is investigated,
namely the determination of all  $W_X$-equivariant
bilinear maps $\pi:M\times N\rightarrow P$
between tensor density modules, without the hypothesis that 
$\pi$ is a differential operator. Since differential
operators are local, we will establish a 
global  (=non-local) version of Grozman Theorem.

For this purpose, we will make new hypotheses. 
From now on, the context is the algebraic geometry,
and the manifold $X$ of investigation is the {\it circle},
namely $\nc^\ast=\Spec \;\nc[z,z^{-1}]$. Set $\bW=W_X$.
Fix two parameters $\delta, s\in \nc$
and set $\rho_{\delta,s}(\xi)=
\xi+\delta\,\div\xi + i_{\xi}\alpha_s$ for any $\xi\in \bW$,
where $\alpha_s=s z^{-1} \d z$. 
By definition,  $\Omega^\delta_s$ is the $\bW$-module
whose underlying space is $\nc[z,z^{-1}]$ and the action
is given by $\rho_{\delta,s}$. To describe more naturally
the action $\rho_{\delta,s}$, it is convenient to denote
by the symbol $z^s (\d z)^\delta$ the generator of this module,
and therefore the expressions $(z^{n+s} (\d z)^\delta)_{n\in\nz}$
form a basis of  $\Omega^\delta_s$. It follows
easily that $\Omega^\delta_s$ and $\Omega^\delta_u$ are
equal if $s-u$ is an integer. Therefore,
we will consider the parameter $s$ as an element
of $\nc/\nz$.   

 We will not provide a rigorous and general  definition of the 
{\it tensor density modules} (see e.g. \cite{M2}). Just say
that the  {\it tensor density modules}
considered here are the $\bW$-modules $\Omega^\delta_{u}$, where $(\delta,s)$ runs over
$\nc\times \nc/\nz$. 

As before, there are $\bW$-equivariant bilinear maps
$P^{\delta_1,\delta_2}_{u_1,u_2}:\Omega^{\delta_1}_{u_1}\times \Omega^{\delta_2}_{u_2}
\rightarrow \Omega^{\delta_1+\delta_2}_{u_1+u_2}$ and
$B^{\delta_1,\delta_2}_{u_1,u_2}:\Omega^{\delta_1}_{u_1}\times \Omega^{\delta_2}_{u_2}
\rightarrow \Omega^{\delta_1+\delta_2 +1}_{u_1+u_2}$, as well
as the de Rham differential
$\d:\Omega^{0}_u\rightarrow\Omega^{1}_u$. There is also a
map $\rho:\Omega^{1}_u\rightarrow\Omega^{0}_u$, which is defined as
follows. For  $u\not\equiv  0\, \mod \nz$, the opeartor $d$ is invertible and
set $\rho=\d^{-1}$. For  $u \equiv 0\, \mod \nz$,
denote by $\rho:\Omega^{1}_u\rightarrow\Omega^{0}_u$
the composite of the residue map
$\Res:\Omega^{1}_0\rightarrow \nc$ and the natural map
$\nc \rightarrow\Omega^{0}_0 =\nc [z,z^{-1}]$. 
By definition, a {\it classical} bilinear map between tensor
density modules over the circle is any linear combination of the operators 
$B^{\delta_1,\delta_2}_{u_1,u_2}$,
$P^{\delta_1,\delta_2}_{u_1,u_2}$ and those obtained by composition with   
$\d$ and $\rho$. An example  of a classical operator is
$\rho\circ P: \Omega^{\delta}_{u_1}\times \Omega^{1-\delta}_{u_2}
\rightarrow \Omega^{0}_{u_1+u_2}$.

Of course, the Grozman operator
provies  a family of non-classical operators
$G_{u,v}: \Omega^{-2/3}_{u}\times \Omega^{-2/3}_{v}
\rightarrow \Omega^{5/3}_{u+v}$. A {\it trivial operator} is 
a scalar multiple of the bilinear map
$\Omega^1_0\times\Omega^1_0 \rightarrow \Omega^0_0, 
(\alpha,\beta)\mapsto\Res(\alpha)\Res(\beta)$.
There is also another
non-classical $\bW$-equivariant  operator 
$\Theta_{\infty}:\Omega^1_0\times
\Omega^1_0\rightarrow \Omega^0_0$ which satisfies:

$\d \Theta_{\infty}(\alpha,\beta)=\Res(\alpha)\beta-\Res(\beta)\alpha$

\noindent for any $\alpha,\beta\in \Omega^1_0$. Indeed $\Theta_{\infty}$ is unique modulo a
trivial operator.

Our result is the following:

\bigskip

{\bf Theorem:} {\it (restricted version) Let $X$ be the circle. Any 
$\bW$-equi\-va\-riant bilinear map between tensor density module is either
classical, or it is a scalar multiple of $G_{u,v}$  or 
of $\Theta_{\infty}$ (modulo a trivial operator).}

\bigskip
 In the paper, a more general version, which also involves deformations of tensor 
density modules, is proved. Set
$L_0=z\dfrac{\d}{\d z}$. For a $\bW$-module
$M$  and $s\in \nc$, set 
$M_s=\Ker (L_0-s)$.
Let $\cS$ be the class of
all $\bW$ modules $M$ which satisfies the following condition:
there exists  $u\in \nc/\nz$ such that

\centerline{$M=\oplus_{s\in u}\,M_s$ and
$\dim M_s=1$ for all $s\in u$.}

\noindent The $\nz$-coset $u$ is called the {\it support of $M$}, and
it is denoted by $\Supp\,M$. 

It turns out that all modules of
the class $\cS$ have been classified by Kaplansky and Santharoubane
\cite{KS})] and, except deformations of 
$\Omega^0_0$ and $\Omega^1_0$, all modules of the class $\cS$
are  tensor density modules. Our full result is the classification of
all $\bW$-equi\-va\-riant bilinear maps between 
modules of the class $\cS$.

\subsection{About the proofs:}\label{sect_0.3}

In order to describe the proof and the organization of the paper,
it is necessary to introduce the notion of germs of bilinear maps.

For any three vector spaces $M,\,N$ and $P$, denote
by $\B(M\times N,P)$  the space of bilinear maps
$\pi:M\times N\rightarrow P$. Assume now that
$M,\,N$ and $P$ are $\bW$-modules of the class $\cS$.
For $x\in{\bf R}$, set
$M_{\geq x}=\oplus_{\Re s\geq x}\, M_s$ and
$N_{\geq x}=\oplus_{\Re s\geq x}\, N_s$. By definition,
a {\it germ} of bilinear map from $M\times N$ to $P$
is an element of the space 
$\cG(M\times N,P):=\lim\limits_{x\rightarrow +\infty}
\,\B(M_{\geq x}\times N_{\geq x},P)$. 

It turns out that  $\cG(M\times N,P)$ is a $\bW$-module.
Denote by $\B_{\bW}(M\times N,P)$  the space of 
$\bW$-equivariant bilinear maps
$\pi:M\times N\rightarrow P$, by 
 $\B^0_{\bW}(M\times N,P)$  the subspace of all
$\pi\in \B_{\bW}(M\times N,P)$ whose germ is zero
and by  $\cG_{\bW}(M\times N,P)$ the space of $\bW$-equivariant 
germs of bilinear maps from $M\times N$ to $P$.

There is a short exact sequence:

\centerline{$0\rightarrow \B^0_{\bW}(M\times N,P)
\rightarrow \B_{\bW}(M\times N,P)
\rightarrow \cG_{\bW}(M\times N,P)$.}

\noindent The paper contains three parts

Part 1 which determines $\B^0_{\bW}(M\times N,P)$, see Theorem 1,

Part 2 which determines $\cG_{\bW}(M\times N,P)$, see Theorem 2,

Part 3 which determines \\
\centerline{$\B_{\bW}(M\times N,P)
\rightarrow \cG_{\bW}(M\times N,P)$,}

\noindent see Theorem 3.

Part 1
is discussed in Section \ref{sect_Thm1}. The  map $\Theta_{\infty}$ is an example of a degenerate map.

Part 2 is the main difficulty of the paper. One checks that
$\dim \cG_{\bW}(M\times N,P)\leq 2$.  So 
it is enough to determine when $\cG_{\bW}(M\times N,P)$ is non-zero
and when   $\dim \cG_{\bW}(M\times N,P)= 2$. 
We will now explain our approach.

The {\it degree}  of the  modules of the class $\cS$ is a 
multivalued function defined as follows. If
$M=\Omega^{\delta}_s$ for some $\delta\neq 0$ or $1$,  set
$\deg M=\delta$. Otherwise, set $\deg M=\{0,1\}$. 
Next, let
 $M, \, N$ and $P\in\cS$ with
$\delta_1\in\deg M$, $\delta_2\in \deg N$ and 
$\gamma\in \deg P$. We can assume that 
$\Supp\,P=\Supp\,M+\Supp\,N$, since otherwise 
$\cG_{\bW}(M\times N,P)$ would be obviously zero.

We introduce a $6$ by $6$ matrix 
$\M=(m_{i,j}(\delta_1,\delta_2,\gamma,x,y))_{1\leq i,j\leq 6}$ whose entries
are quadratic polynomials in the five variables $\delta_1,\delta_2,\gamma,x,y$
and which satisfies the following property:

\centerline{$\det \M=0$ for all $x, y$
if $\cG_{\bW}(M\times N,P)\neq 0$.} 

Set 
$\det \M=\sum_{i,j}\,p_{i,j}(\delta_1,\delta_2,\gamma) x^iy^j$
and let $\fZ$ be the common set of zeroes of
all polynomials $p_{i,j}$. 
Half of the entries of $\M$ are zero and only $16$
of the $720$ diagonals of $\M$  give a non-zero contribution
for $\det\M$. However  a human  computation
looks too complicated, because each non-zero entry of $\M$ is a linear combination
of  $9$ or $10$ distinct monomials. The computation of  the polynomials $p_{i,j}$
has been done with MAPLE.

As expected, $p_{1,3}$ and $p_{3,1}$ are degree eight polynomials.
It turns out that each of them is a product of $6$ degree one factors and one quadratic factor.
Indeed 4 degree one factors are obvious and the rest of
the factorizations look miraculous.
Moreover the two (suitably normalized) quadratic factors differ by 
a linear term. It follows that the common zero set of $p_{1,3}$ and $p_{3,1}$
is a union of affine planes, affine lines and some planar quadrics. 
This allows to explicitely solve the equations $p_{i,j}=0$. 
Since only the polynomials $p_{1,3}, p_{3,1}$ and $p_{2,2}$ are needed,
the other polynomials $p_{i,j}$ are listed in Appendix A.
It turns out that $\fZ$ decomposes into four planes,  eight lines and four points. 

Using an additional trick,
we determine when  $\cG_{\bW}(M\times N,P)\neq 0$, and when
$\dim \cG_{\bW}(M\times N,P)=2$. 
Although its proof is the main difficulty 
of the paper, the statement of Theorem 2 is very simple.
Indeed $\cG_{\bW}(M\times N,P)$ is non-zero exactly when  
$(\delta_1,\delta_2,\gamma)$ belongs to an explicit algebraic subset $\fz$ of $\fZ$ consisting
of two planes, six lines and five points. Moreover, it   has dimension two
iff  $\{\delta_1,\delta_2,\gamma\} \subset\{0,1\}$).

Theorem 3 determines which  germs in $\cG_{\bW}(M\times N,P)$ can be lifted to a
$\bW$-equivariant bilinear map. Each particular case is easy,  but the list  is
very long. Therefore Theorem 3 has been split into Theorem 3.1 and Theorem 3.2, corresponding to the case  where $\cG_{\bW}(M\times N,P)$ has dimension one or two.

It should be noted that all $\bW$-modules of the class $\cS$ are indecomposable except one,
namely $\overline A\oplus \nc$, where $\overline A=\nc[z,z^{-1}]/\nc$. In most statements
about bilinear maps $\pi:M\times N\rightarrow P$, we assume that $M$, $N$ and $P$ are
indecomposable. Indeed the case where some modules are decomposable follows easily. The indecomposability hypothesis removes many less interesting cases. This is
helpful since some statements already contain many particular cases, e.g. Theorem 3.1 
contains 16 of them.

\begin{ack}
We would like to thank our colleague J\'er\^ome Germoni for his computation
of the determinant $\det M$ with the aid of MAPLE.
\end{ack}

\section{The Kaplansky-Santharoubane Theorem}\label{KS_Theorem}
The {\it Witt algebra} $\bW$ is the Lie algebra of derivations of the Laurent 
polynomial ring $A =\nc[z,z^{-1}]$. 
Clearly the elements $L_n=z^{n+1}\dfrac{d}{dz}$, where $n$ runs over $\nz$,
form a basis of $\bW$ and we have
 
\noindent\centerline{$ [L_m, L_n]=(n-m)L_{m+n}$.}

\noindent Throughout this paper,  $\fsl(2)$ refers to its
subalgebra

\noindent\centerline{$\nc L_{-1}\oplus\nc L_0\oplus\nc L_1$.}

\subsection{Statement of the theorem}\label{sect_KS}
For a $\bW$-module $M$, set $M_z=\{m \in M \vert L_0.m=zm\}$ for any $z \in
\nc$ and define its {\it support} as the set $\Supp M=\{z \in \nc
\vert M_z\neq 0\}$.

Let $\cS$ be the class of all $\bW$-modules $M$
such that 

(i) $M=\oplus_{z \in \nc} M_z$,

(ii) $\Supp M$ is exactly one $\nz$-coset, and

(ii) $\dim M_z=1$ for all $z \in \Supp (M)$.

\noindent Here are three families  of modules of the class $\cS$:

\begin{enumerate}
\item The family of  {\it tensor density modules} $\Omega_{u}^\delta$,
where $(\delta,u)$ runs over $\nc\times \nc/\nz$. Here
$\Omega_{u}^\delta$ is the $\bW$-module with basis 
$(e_z^\delta)_{z\in u}$ and  action given by the formula:

\hskip3cm $ L_m.e_z^\delta=(m\delta+z)e_{z+m}^\delta.$

\item 
The {\it $A$-family} $(A_{a,b})_{(a,b)\in\nc^2}$. Here 
$A_{a,b}$ is  the $\bW$-module with basis 
$(e_n^A)_{n\in\nz}$ and action given by the formula:

\hskip3cm$ L_m.e_n^A=\begin{cases} (m+n)e_{m+n}^A \qquad & n \neq 0, \\
 (am^2+bm)e_m^A \qquad &n=0. \end{cases} $

\item 
The {\it $B$-family} $(B_{a,b})_{(a,b)\in\nc^2}$. Here $B_{a,b}$ is the
$\bW$-module with basis $(e_n^B)_{n\in\nz}$ and action given by the
formula:

\hskip3cm $L_m.e_n^B=\begin{cases} ne_{m+n}^B \qquad &n+m \neq 0, \\
(am^2+bm)e_0^B \qquad & n+m=0. \end{cases}$
\end{enumerate}

Set $\overline{A}:=A/\nc$. There are two exact sequences:

\noindent\centerline{
$0 \longrightarrow  \overline{A} \longrightarrow A_{a,b} 
{\longrightarrow} \nc \longrightarrow 0$,  and}

\noindent\centerline{
$0 \longrightarrow   \nc {  \longrightarrow} B_{a,b} \longrightarrow
\overline{A}
\longrightarrow 0$,}

\noindent and we denote by $\Res: A_{a,b}\rightarrow\nc$  the  map defined by
$\Res\, e_0^A=1$ and  $\Res\, e_n^A=0$ if $n\neq 0$.
These exact sequences
do not split, except for
$(a,b)=(0,0)$. Therefore the $A$-family is a deformation of
$\Omega^1_0\simeq A_{0,1}$ and the $B$-family is a deformation of of 
$\Omega^0_0\simeq B_{0,1}$.
Except the previous two isomorphisms and the obvious
$A_{0,0}\cong B_{0,0} \cong \overline{A} \oplus \nc$, there are some repetitions in  the 
previous
list due to the following isomorphisms:

\begin{enumerate}

\item the de Rham differential $d:\Omega_{u}^0 \rightarrow \Omega_{u}^1$, 
if $u\not\equiv 0$ mod $\nz$,

\item $A_{\lambda a, \lambda b} \cong A_{a,b}$ and 
$B_{\lambda a, \lambda b} \cong B_{a,b}$ for $\lambda \in \nc^\ast$, 

\end{enumerate}

\noindent 
There are no other isomorphism in the class $\cS$ beside those previously indicated. 
 From now on, we will consider the couples $(a,b)\neq (0,0)$
as a projective coordinate, and  the 
indecomposable modules in the $AB$-famillies are now parametrized by
$\np^1$. Set  $\infty=(0,1)$ and
$\nA^1=\np^1\setminus \infty$. Therefore the indecomposable $\bW$-modules in 
the previous list, which are not tensor density modules, are the two $\nA^1$-parametrized
families  $(A_\xi)_{\xi\in\nA^1}$ and $(B_\xi)_{\xi\in\nA^1}$,
as in \cite{MP}'s paper.

 The classification of the $\bW$-modules of the class $\cS$ was 
given by I. Kaplansky and L. J. Santharoubane \cite{KS}, \cite{K} (with
a minor correction in  \cite{MP} concerning the parametrization of the
$AB$-families):

\bigskip
{\bf Kaplansky-Santharoubane Theorem.}
{\it Let $M$ be a $\bW$-module of the class $\cS$.
\begin{enumerate}
\item If $M$ is irreducible, then there exists
$(u, \delta) \in \nc/\nz \times \nc$, with $(u,\delta)\neq (0,0)$ or $(0,1)$ such that 
$M \simeq\Omega_{u}^\delta$.

\item If $M$ is reducible and indecomposable, then $M$ is isomorphic to either 
$A_{\xi}$ or $B_{\xi}$ for some $\xi \in \np^1$. 

\item Otherwise, $M$ is isomorphic to $\overline{A}\oplus \nc$.
\end{enumerate}
}

\subsection{Degree of the modules in the class $\cS$}
It follows from the previous remark that one can define 
{\it the degree} $\deg M$ for any $M \in \cS$ as follows:

\begin{enumerate}
\item $\deg M=\delta$ if $M \cong \Omega_s^\delta$ for some 
$\delta \in \nc \setminus \{0,1\}$, and

\item $\deg M=\{0,1\}$ otherwise.
\end{enumerate}

By definition, the degree is a multivalued function. We also define
{\it a degree} of $M$ as a value $\delta \in \deg M$. Let $\cS^\ast$
be the class of all pairs $(M,\delta)$, where $M\in \cS$ and $\delta \in \deg
M$. A pair $(M,\delta) \in \cS^\ast$  will be often simplified as $M$ and
set $\deg M:=\delta$. So, the degree function is a single valued function
on $\cS^\ast$. Usually,  we consider 
$\Omega_s^\delta$ as the element $(\Omega_s^\delta,\delta)$ of $\cS^\ast$
for any $\delta$.

For $M\in\cS$,  let $M^\ast$ be its restricted dual, namely
$M^\ast=\oplus_{x \in \nc} M_x^\ast$. 
By definition, the class $\cS$ is stable by the restricted duality and  
we have:

\begin{lemma}\label{lemma_dual-density} 
$(\Omega_u^\delta)^\ast \cong \Omega_{-u}^{1-\delta}$ and $(A_{\xi})^\ast \cong
B_{\xi}$.
\end{lemma}

In particular, it follows that $\deg M^\ast=1-\deg M$ for any 
$M\in \cS$.

\section{Germs and bilinear maps}

\subsection{On the terminology `$\fg$-equivariant'}

Throughout the whole paper, we will use the following convention.
Let $\fg$ be a Lie algebra and let $E$ be a $\fg$-module.
When $E$ is a space of maps, a $\fg$-invariant element of
$E$ will be called {\it $\fg$-equivariant}. We will use the same convention
for spaces of germs of maps (see the definition below).

When $\fg$ is the one-dimensional Lie algebra $\nc.L$, we will use the
terminology {\it $L$-invariant} and {\it $L$-equivariant} instead of
$\fg$-invariant and $\fg$-equivariant.

\subsection{Weight modules and the $\fS_3$-symmetry}

 A {\it $\nc$-graded vector space} is a vector space $M$ endowed with a decomposition
$M=\oplus_{z\in\nc} M_z$ such that $\dim M_z<\infty$ for all
$z\in \nc$. Denote by $\cH$  the category of all
$\nc$-graded vector spaces. It is convenient
to denote by $L_0$ the degree operator, which acts as $z$ on
$M_z$. Given $M,N\in{\cal H}$, we denote by
$\Hom_{L_0}(M,N)$ the space of $L_0$-equivariant linear maps from $\phi:M\rightarrow N$.
Equivalently $\Hom_{L_0}(M,N)$ is the space of maps in the category ${\cal H}$.

By definition, a {\it Lie $L_0$-algebra} is a pair $(\fg,L_0)$, where  
$\fg=\oplus_{z\in\nc}\,\fg_z$ is a $\nc$-graded Lie algebra,  
$L_0$ is an element of $\fg_0$ such that $\ad(L_0)$ acts as the degree operator.  
 A {\it weight} $\fg$-module is a 
$\nc$-graded vector space $M$ endowed with a 
structure of $\fg$-module.
Of course it is required that $L_0$
acts as the degree operator on $M$ and therefore we have
$\fg_y.M_z\subset M_{y+z}$ for all $y,\,z\in\nc$. 
Let $\cH_{\fg}$ be the category of weight $\fg$-modules.
For $M$ and $N$ in $\cH_{\fg}$, denote by $\Hom_{\fg}(M,N)$
the space of $\fg$-equivariant linear maps from $M$ to $N$.

Given $M,\,N$ and $P$ in $\cH$, denote by $\B(M\times N,P)$ the space 
of all bilinear maps $\pi:M\times N\rightarrow P$ and by 
$\B_{L_0}(M\times N,P)$ the subspace of $L_0$-equivariant bilinear maps.
Similarly if $M,\,N$ and $P$ are in $\cH_{\fg}$, 
denote by $\B_{\fg}(M\times N,P)$
the space of $\fg$-equivariant bilinear maps.

For $P\in\cH$, denote by $P^*$  its restricted dual. By definition, we have
$P^*=\oplus_{z\in \nc}\, P^*_z$, where $P^*_z=(P_{-z})^*$.

\begin{lemma}\label{lemma_symmetry-bilinear}
Let $M,\,N$ and $P$ in ${\cH}_{\fg}$. We have:

\noindent
\centerline{$\B_{\fg}(M \times N, P^\ast)  \simeq \B_{\fg}(M \times P, N^\ast)$.} 
\end{lemma}

\begin{proof} The lemma follows easily from the fact that

\centerline{$\B_{L_0}(M\times N,P^*)=\prod_{u+v+w=0}\,M^*_u
\otimes N^*_v\otimes P^*_w$.}
\end{proof}

It follows that $\B_{\fg}(M\times N,P^\ast)$ is fully symmetric in $M,N$ and
$P$. This fact will be referred to as the {\it $\fS_3$-symmetry}. The obvious
symmetry $\B_{\fg}(M\times N,P)\simeq \B_{\fg}(N\times M,P)$
will be called the  {\it $\fS_2$-symmetry}

\subsection{Definition of germs}

For  $M\in\cH$ 
and $x\in \nr$, set 
$M_{\geq x}=\oplus_{\Re z\geq x}\,M_z$ and
$M_{\leq x}=\oplus_{\Re z\leq x}\,M_z$,
where $\Re z$ denotes the real part of $z$. 
Given another object 
$N\in\cH$, let $\Hom^0(M,N)$ be the space of all linear maps
$\phi:M\rightarrow N$ such that
$\phi(M_{\geq x})=0$ for some $x\in{\bf R}$. Set
$\cG(M,N)=\Hom(M,N)/\Hom^0(M,N)$. 
The image  in $\cG(M,N)$ 
of some $\phi\in \Hom(M,N)$, which is denoted by $\cG(\phi)$,
is called its {\it germ}. The space $\cG(M,N)$ is called the 
{\it space of germs of maps} from $M$ to $N$.

Let $\fg$ be a Lie $L_0$-algebra and
let $M,\,N\in{\cH}_{\fg}$. It is clear that $\Hom^0(M,N)$ is a 
$\fg$-submodule of $\Hom(M,N)$ and thus $\fg$ acts on
$\cG(M,N)$. Denote by
$\cG_{\fg}(M,N)$ the space of $\fg$-equivariant germs. 
 We will often use the following
obvious fact: any $\psi\in\cG_{L_0}(M,N)$ is the germ of
a $L_0$-equivariant map $\phi:M\rightarrow N$, but
in general a $\fg$-equivariant germ $\psi$ is not the germ of a 
$\fg$-equivariant map $\phi$.

 Let $M,\,N\in\cH$.
A linear map  $\phi:M\rightarrow N$ is called {\it continous}
if for any $x\in\nr$ there exist $y\in\nr$ such that
$\phi(M_{\geq y})\subset N_{\geq x}$. The germ of 
a continous map $\phi$ is called a {\it continous germ} of a map.

 It is not possible to compose arbitrary germs of maps.
However let  $\phi,\psi$ two morphisms of $\cH$ 
such that the composition $\psi\circ\phi$ is defined.
It is easy to show that $\cG(\psi\circ\phi)$ only depends  
of  $\cG(\phi)$ and $\cG(\psi)$ whenever $\phi$ is continous. Thus,
it is possible to compose the continous germs.

Since the $L_0$-equivariant germs of maps are continous, 
the $\fg$-equivariant germs can be composed. 
Therefore we can define the category $\cG(\cH_\fg)$ of {\it germs of
weight $\fg$-modules} as follows. Its objects are weight $\fg$-modules, and  
for $M,\,N\in\cH_\fg$, the space of 
$\cG(\cH_\fg)$-morphisms from $M$ to $N$ is 
$\cG_{\fg}(M,N)$. Viewed as an object of the
category $\cG(\cH_\fg)$, an object   $M\in\cH_\fg$ is called a
{\it germ of a weight $\fg$-module} and it is denoted
by $\cG_{\fg}(M)$.

When $\fg_{\geq 0}$ and $\fg_{\leq 0}$ are finitely generated as Lie algebras,
there is a concrete  characterization of the $\fg$-equivariant germs of maps.
Indeed let $M,N\in\cH_{\fg}$ and let $\phi:M\rightarrow N$ be a $L_0$-equivariant
map. Then $\cG(\phi)$ is $\fg$-equivariant iff: 

(i) the restriction $\phi:M_{\geq x}\rightarrow N_{\geq x}$ is
$\fg_{\geq 0}$-equivariant, and

(ii) the induced map $\phi:M/M_{\leq x}\rightarrow N/N_{\leq x}$ is
$\fg_{\leq 0}$-equivariant,

\noindent for any $x>>0$.

\subsection{Germs of modules of the class $\cS$}

It is easy to compute the germs of the $\bW$-modules of the class 
$\cS$. 

\begin{lemma}\label{W-germs}
For any $\xi_1,\xi_2\in \np^1$, we have
$\cG_{\bW}(A_{\xi_1})=\cG_{\bW}(B_{\xi_2})$.
Thus for any $M\in\cS$, we have
${\cG}_{\bW}(M)\simeq{\cG}_{\bW}(\Omega^\delta_{u})$ for some
$\delta\in\nc$ and $u\in\nc/\nz$.
\end{lemma}

The proof of the lemma  follows easily from the definition. 
Recall that $\fsl(2)$ is the Lie subalgebra
$\nc L_{-1}\oplus \nc L_{0}\oplus \nc L_{1}$ of
$\bW$.  In what follows, it is useful to compare  $\fsl(2)$-germs
and $\bW$-germs.

\begin{lemma} \label{sl(2)-germs}
Let $\delta\in\nc$ and $u\in\nc/\nz$. 

(i) We have 
$\cG_{\fsl(2)}(\Omega^{\delta}_u)\simeq
\cG_{\fsl(2)}(\Omega^{1-\delta}_u)$.

(ii) If  $\cG_{\bW_{\geq -1}}(\Omega^{\delta}_u)\simeq
\cG_{\bW_{\geq -1}}(\Omega^{\gamma}_u)$  for some 
$\delta\neq \gamma$, then $\{\delta,\gamma\}=\{0,1\}$.
\end{lemma}

\begin{proof} 
{\it Proof of the first assertion:}

Choose a function $f:\nc\rightarrow\nc^*$ such that
$f(z)=zf(z-1)$ whenever $\Re z>1$. Define
$\psi:\Omega^{\delta}_u\rightarrow\Omega^{1-\delta}_u$ by the formula:

\noindent\centerline{$\psi(e^\delta_z)=f(z-\delta)/f(z+\delta-1)\,
e^{1-\delta}_z$,}

\noindent for any $z\in u$. 
It is easy to check that
$L_{\pm1}.\psi(e^\delta_z)=\psi(L_{\pm1}.e^\delta_z)$ whenever
$z\in u$ and
$\Re (z\pm\delta)>1$. Therefore  the germs of $L_{\pm1}.\psi$ are zero,
which means that $\cG(\psi)$ is a $\fsl(2)$-equivariant 
isomorphism.

{\it Proof of the second assertion:}
Assume that the $\bW_{\geq -1}$-germs 
of $\Omega^{\delta}_u$ and $\Omega^{\gamma}_u$ are
isomorphic. Then there exist a $L_0$-equivariant isomorphism
$\psi:\Omega^{\delta}_u\rightarrow\Omega^{\gamma}_u$
whose germ is $\bW_{\geq -1}$-equivariant. It follows that

$\psi(L_1^2. e^\delta_z)=L_1^2.\psi(e^\delta_z)$ and
$\psi(L_2. e^\delta_z)=L_2.\psi(e^\delta_z)$,

\noindent 
for any $z\in u$ with $\Re z>>0$. 
Set $\psi(e^\delta_{z})=ae^{\gamma}_{z}$
and $\psi(e^\delta_{z+2})=be^{\gamma}_{z}$.
It follows that:

$(z+\delta)(z+1+\delta)b =(z+\gamma)(z+1+\gamma)a$, and

$(z+2\delta) b=(z+2\gamma)a$.

Therefore we get:

$(z+\delta)(z+1+\delta)(z+2\gamma)=
(z+\gamma)(z+1+\gamma)(z+2\delta)$. 

\noindent Since this identity holds for any $z\in u$ with $\Re z>>0$,
it is valid for any $z$. Since $\delta\neq\gamma$, it follows easily
that  $\{\delta,\gamma\}=\{0,1\}$.
\end{proof}

 It follows from Lemma \ref{sl(2)-germs}(ii) that 
the degree  of modules of the class $\cS$ is  indeed an invariant of 
their $\bW$-germs.

\subsection{Germs of bilinear maps}
For  $M,N$ and $P\in\cH$, denote by $\B(M\times N,P)$ the
space of bilinear maps from $M\times N$ to $P$. Also denote by 
$\B^0(M\times N,P)$  the space of all $\pi\in \B(M\times N,P)$ such that 
$\pi(M_{\geq x}\times N_{\geq x})=0$ for any $x>>0$. Set 

\centerline{$\cG(M\times N,P)=\B(M\times N,P)/\B^0(M\times N,P)$,}

The image of a bilinear map $\pi\in \B(M\times N,P)$ in
$\cG(M\times N,P)$ is called its {\it germ} and it is denoted by
${\cG}(\pi)$. The set  $\cG(M\times N,P)$ is called the {\it space of germs of bilinear
maps} from $M\times N$ to $P$.

Let $\fg$ be a Lie $L_0$-algebra and let 
$M,\,N$ and $P$ be weight $\fg$-modules. As before,
$\cG(M\times N,P)$ is a $\fg$-module in a natural way
and we denote by $\cG_{\fg}(M\times N,P)$  the
space of $\fg$-equivariant germs of bilinear maps.
As before,the composition of $\fg$-equivariant germs of 
bilinear maps with  $\fg$-equivariant germs of 
linear maps is well defined. Thus we obtain:

\begin{lemma}\label{lemma_germ-gen}
The space  $\cG_{\fg}(M\times N,P)$  depends functorially  on the germs of the
weight $\fg$-modules $M$, $N$ and $P$.
\end{lemma}

Let $M,N,P$, and $Q$ in $\cS$ and let $\phi\in \Hom_{L_0}(P,Q)$.
Assume that $\cG(\phi)$ is a $\fsl(2)$-equivariant isomorphism. The composition 
with $\phi$ induces a map

\noindent\centerline{$\cG(\phi)_*:\cG_{_{\fsl(2)}}(M\times N,P)\rightarrow
\cG_{_{\fsl(2)}}(M\times N,Q)$.}

\begin{lemma}\label{zero_intersection}
Assume that $\cG(\phi)$ is not $\bW_{\geq -1}$-equivariant. Then  the two subspaces 
$\cG_{_\bW}(M\times N,Q)$ and 
$\cG(\phi)_*\cG_{\bW}(M\times N,P)$ of 
$\cG_{_{\fsl(2)}}(M\times N,Q)$ have a zero intersection.
\end{lemma}

\begin{proof}
 The lemma is equivalent to the following statement:
for any  $L_0$-equivariant bilinear map 
$\pi:M\times N\rightarrow P$ whose germ is
$\bW$-equivariant and non-zero,  $\cG(\phi\circ\pi)$ is not
$\bW$-equivariant. So, we prove this statement. 

Set $\mu=L_2.(\phi\circ\pi)$.
We claim that $\cG(\mu)\neq 0$, i.e. for any $r\in\nr$ there are scalars 
$x,y$ with $\Re x>r$ and $\Re y>r$ such that $\mu(M_x\times N_y)\neq 0$
Indeed, if  $r$ is big enough  we have

(i) the restriction 
$\pi_{\geq r}: M_{\geq r}\times N_{\geq r}\rightarrow P_{\geq 2r}$ of
$\pi$ is ${\bW}_{\geq 0}$-equivariant, and

(ii)$L_1.P_z=P_{z+1}$ for any $z$ with $\Re z>2r$.

\noindent Since  $\cG(\pi)\neq 0$, there exists $(x_0,y_0)\in\Supp M\times\Supp N$
with $\Re x_0>r,\, \Re y_0>r$ and $\pi(M_{x_0}\times N_{y_0})=P_{z_0}$, where 
$z_0=x_0+y_0$. 

By hypothesis $\cG(\phi)$ is $\fsl(2)$-equivariant but not
$\bW_{\geq -1}$-equivariant. Since $\bW_{\geq -1}$ is generated by
$\fsl(2)$ and $L_2$, we have $\cG(L_2.\phi)\neq 0$. Hence there exists
$k\in{\bf Z}_{\geq 0}$ such that $(L_2.\phi)(P_{k+z_0})\neq 0$. 

By assumptions, the linear span of $\cup_{m.n\geq 0}\,\pi(M_{x_0+m}\times N_{y_0+n})$ is
a $\bW_{\geq 0}$-module and the $\bW_{\geq 0}$-module $P_{\geq \Re z_0}$ is generated
by $P_{z_0}=\pi(M_{x_0}\times N_{y_0})$. Thus there are $m,n\in\nz_{\geq 0}$
with $m+n=k$ such that $\pi(M_{x_0+m}\times N_{y_0+n})=P_{z_0+k}$.

Since $L_2.(\phi\circ\pi_{\geq r})
=(L_2.\phi)\circ\pi_{\geq r}$, it follows that
$\mu(M_{x_0+m}\times N_{y_0+n})\neq 0$, which proves the claim.
 \end{proof}

\section{Degenerate and non-degenerate \\
 bilinear maps}\label{sect_(non)deg-map} 

In this section, we define the notions of {\it degenerate} and
{\it non-degenerate} bilinear maps and similar notions for
germs of bilinear maps. We show that a
$\bW$-equivariant bilinear map $\pi$ between modules of the class $\cS$ is degenerate if and
only if $\cG(\pi)=0$. Moreover,  $\cG(\pi)$ is non-degenerate if $\cG(\pi)\neq 0$.

Let $M,N$ and  $P$ in $\cH$. For $\pi\in \B(M\times N,P)$, the set
$\Supp\,\pi=\{(x,y)\vert\,\pi(M_x\times N_y)\neq 0\}$ 
is called the {\it   support} of $\pi$.
The bilinear map $\pi$ is called {\it   non-degenerate} if
$\Supp\,\pi$ is Zarisky dense in $\nc^2$. Otherwise, it is called
{\it degenerate}.

Any germ $\tau \in \cG(M\times N,P)$ is represented by a
bilinear map $\pi \in \B(M\times N,P)$, and let $\pi_{\geq x}$ be its restriction
to $M_{\geq x} \times N_{\geq x}$.
The germ $\tau$ is called 
{\it non-degenerate} if $\pi_{\geq x}$ is non-degenerate for any $x>>0$.

From now on, assume that $M$, $N$ and $P$ are $\bW$-modules of
the class $\cS$.  For $\pi\in \B_{\bW}(M\times N,P)$, set
$M_\pi=\{m\in M\vert\,\pi(m\times N_{\geq x})=0\,
\text{for}\; x>>0\}$ and
$N_\pi=\{n\in N\vert\,\pi(M_{\geq x}\times n)=0\,
\text{for}\; x>>0\}$.
It is clear that $M_\pi$ and $N_\pi$ are $\bW$-submodules.

\begin{lemma}\label{lemma5}
Let $\pi\in \B_{\bW}^0(M\times N, P)$. Then we have:

\begin{enumerate}
\item[(i)] $\pi(M_\pi\times N_\pi)\subset P^{\bW}$ and therefore
$\pi(M_\pi\times N_\pi)$ is isomorphic to 0 or $\nc$.

\item[(ii)] $M/M_\pi$ is isomorphic to 0 or $\nc$.

\item[(iii)] $N/N_\pi$ is isomorphic to 0 or $\nc$.
\end{enumerate}
\end{lemma}

\begin{proof}
It follows from the explicit description of all modules $X\in\cS$ 
(see Section \ref{KS_Theorem}) that

\begin{enumerate}
\item if $Y$ is a $\bW$-submodule with $L_0.Y\neq 0$, then $X/Y$ is isomorphic to $\nc$ or
$0$, and

\item if $x \in X$ satisfies $L_k.x=0$ for $k>>0$, then $x$ is 
$\bW$-invariant.
\end{enumerate}
Since $\cG(\pi)$ is zero, 
 $M_\pi$ contains $M_{\geq x}$ for any $x>>0$. 
Therefore, $L_0.M_\pi\neq 0$ and Assertions (ii) and (iii) follows.

Moreover for any $(m,n)\in M_\pi\times N_\pi$, we have
$L_k.\pi(m,n)=\pi(L_k.m,n)+\pi(m,L_k.n)=0$ for $k>>0$. Thus $\pi(m,n)$ is $\bW$-invariant which proves the first
assertion.
\end{proof}

Let  $\pi\in \B_{\bW}(M\times N, P)$ with
$\Supp\,\pi\subset \{(0,0)\}$. Obviously there are
$\bW$-equivariant maps $a:M\rightarrow \nc$,
$b:N\rightarrow \nc$ and $c:\nc\rightarrow P$ such that
$\pi(l,m)=c(a(l)b(m))$. Since it comes from a bilinear
map between trivial modules, such a bilinear map $\pi$ will be
called {\it   trivial}. Note that non-zero trivial maps only occur when
$M$ and $N$ are in the $A$-family and $P$ is in the $B$-family.

For a subset $Z$ of $\nc^2$, denotes
by $\overline{Z}$ its Zariski closure. Also
define the three lines $H$, $V$ and $D$ of $\nc^2$ by:

\noindent\centerline{$H=\nc \times \{0\}$, 
$V=\{0\} \times \nc$ and
$D=\{(z,-z)\vert\, z\in \nc \}$.}

\begin{lemma}\label{lemma_support}
Let $\pi\in \B_{\bW}^0(M\times N, P)$
and  set $S=\overline{\Supp \,\pi}$. Assume that $\pi$ is not trivial. Then
we have:
\begin{enumerate}
\item[(i)] $S$ is a union of lines and $\Supp\,\pi\subset H\cup D\cup V$.
\item[(ii)] $\pi(M_\pi\times N_\pi)\neq 0$
iff $D\subset S$.
\item[(iii)] $M/M_\pi\neq 0$ iff $V\subset S$.
\item[(iv)] $N/N_\pi\neq 0$ iff $H\subset S$.
\end{enumerate}
\end{lemma}

\begin{proof} By Lemma \ref{lemma5}, we have 
$\pi(M_\pi\times N_\pi)=0$ or $\nc$.
Note that $\pi$ induces the two  bilinear maps 
$\eta: M_\pi\times N_\pi\rightarrow \pi(M_\pi\times N_\pi)$ and 
$\theta: M\times N\rightarrow P/\pi(M_\pi\times
N_\pi)$.

{\it Step 1:} We claim that $\eta=0$ or the bilinear map $\eta$ has
infinite rank. Assume otherwise. By Lemma \ref{lemma5}, the image of $\eta$ is $\nc$. 
Since $\eta$ factors through finite
dimensional modules, it follows that $M_\pi$ has a finite 
dimensional quotient. By Lemma \ref{lemma5} (ii), $M_\pi$ is infinite dimensional. 
Hence  $M_\pi$ is reducible, which implies that  $M=M_\pi$.
Similarly, $N=N_\pi$. It follows easily that $\pi$ is a trivial bilinear map which contradicts
the hypothesis. 

It follows that $\eta=0$ or
$\overline{\Supp\, \eta}=D$. 

{\it Step 2:}
We claim that $\overline{\Supp\,\pi}=
\overline{\Supp\, \theta}\cup \overline{\Supp\, \eta}$.

Since $\pi(M_\pi\times N_\pi)\simeq \nc$ or $0$, we have:
$\overline{\Supp\, \theta}\cup \overline{\Supp\, \eta}
\subset \overline{\Supp\,\pi}\subset
\overline{\Supp\, \theta}\cup D$.
Therefore the claim follows from the previous step.

{\it Step 3:} Using the short exact sequence
\begin{align*}
0 \longrightarrow & M\otimes N/M_\pi\otimes N_\pi
\longrightarrow M\otimes (N/N_\pi)\oplus (M/M_\pi)\otimes N \longrightarrow \\
&\longrightarrow  M/M_\pi \otimes N/N_\pi \longrightarrow 0,
\end{align*}
it follows that, up to the point $(0,0)$, the sets
$\Supp\, \theta$ and
$\Supp\, M\times \Supp\,(N/N_\pi)\\ \cup \Supp\,(M/M_\pi)\times\Supp\, N$
coincide, which proves the lemma.
\end{proof}

\begin{lemma}\label{germ-deg}
A bilinear map $\pi\in \B_{\bW}(M\times N, P)$ is degenerate iff 
its germ is zero. 
Moreover, if $\cG(\pi)\neq 0$, the germ $\cG(\pi)$ is non-degenerate. 
\end{lemma}

\begin{proof}
Let $\pi \in \B_{\bW}(M\times N, P)$. 
For $x\in\nr$, denote by $\pi_{\geq x}$
the restriction of $\pi$ to $M_{\geq x}\times N_{\geq x}$.
It is enough to prove the second assertion.

{\it First Step:}  We claim that if
$(s,t)$ belongs to $\Supp\,\pi_{\geq 1}$, then $(s,t)+H$ or $(s,t)+V$ lies in 
$\overline{\Supp\, \pi_{\geq 1}}$. By hypothesis, $\Re (s+t)>0$ and it follows from the 
explicit description of modules of the class $\cS$ that
$L_k.P_{s+t}\neq 0$ for any $k>>0$. So we get

\noindent\centerline{$\pi(L_k.M_s\times N_t)\neq 0$ or
$\pi(M_s\times L_k.N_t)\neq 0$ for $k>>0$.}

Hence $(s+k,t)$ or $(s,t+k)$ is in $\Supp\,\pi_x$ for infinitely many
$k>0$, and the claim follows.

{\it Second step:} Assume that $\cG(\pi)$ is non-zero and prove that
its germ is non-degenerate. Let $x\geq 1$ be an arbitrary real number, and let
$(s,t)\in\Supp\, M_{\geq x}\times\Supp\, N_{\geq x}$. By definition
there exists  two increasing sequences of
integers $0\leq a_1<a_2\dots$ and $0 \leq b_1<b_2\dots$  such that 
$(s+a_k, t+b_k)$ belongs to $\Supp\,\pi_{\geq x}$ for all $k$. Since all lines 
$(s+a_k, t+b_k)+V$, $(s+a_k, t+b_k)+H$ are distinct, 
$\overline{\Supp\,\pi_{\geq x}}$ contains infinitely many lines, and therefore
$\overline{\Supp\,\pi_{\geq x}}=\nc^2$
\end{proof}

\section{Examples of $\bW$-equivariant bilinear maps}\label{sect_example}

This section provides a list of $\bW$-equivariant bilinear maps between modules of the class $\cS$. 
The goal of this paper is to prove that this list generates all 
bilinear maps between modules of the class $\cS$. 
More precisely, if one allows the following operations:
the $\fS_3$-symmetry, the composition with morphisms between modules in $\cS$ and
the linear combination, then one obtains all  
bilinear maps between modules of the class  $\cS$.

\subsection {The Poisson algebra $\cP$ of symbols twisted pseudo-differential operators}

To be brief, we will not give the  definition of the algebra $\cD$  of twisted pseudo-differential
operators on the circle, see e.g \cite{IM}.  Just say that 
the term {\it twisted} refers to the fact that complex powers of $z$ and   $\D$ are allowed.
As usual, $\cD$ is an associative filtered 
algebra whose associated graded space  ${\cP}$ is a Poisson algebra.

Indeed ${\cP}$ is explicitely defined as follows. As vector space, ${\cP}$ has basis the family
$(z^s\partial^\delta)$ where $s$ and
$\delta$ runs over  $\nc$ (here $\partial$ stands for the symbol of
$\D$).   The commutative associative product on  ${\cP}$
is denoted by $.$ and the Lie bracket is denoted by $\{,\}$. These products are  explicitely 
defined on the basis elements by

\noindent\centerline{$(z^s\partial^\delta).(z^{s'}\partial^{\delta'})=
(z^{s+s'}\partial^{\delta+\delta'})$,  }

\noindent\centerline{$\{(z^s\partial^\delta),(z^{s'}\partial^{\delta'})\}
=(\delta s'-\delta' s)\,z^{s+s'-1}\partial^{\delta+\delta'-1}$.}

\noindent It is clear that $\oplus_{n\in{\nz}}\,\nc\, z^{n+1}\partial$ is a Lie subalgebra 
naturally isomorphic to $\bW$. 
As a $\bW$-module, there is a decomposition  of $\cP$ as

\noindent\centerline{
$\cP=\oplus_{(\delta,u)\in \nc\times\nc/\nz}\,\,\Omega_u^{\delta}$,
}

\noindent where $\Omega_u^{\delta}=\oplus_{s\in u}\,\nc\,z^{s-\delta}\partial^{-\delta}$.
We have

\noindent\centerline{
$\Omega_u^{\delta_1}.\Omega_v^{\delta_2}\subset \Omega_{u+v}^{\delta_1+\delta_2}$,
}

\noindent\centerline{
$\{\Omega_u^{\delta_1},\Omega_v^{\delta_2}\}\subset \Omega_{u+v}^{\delta_1+\delta_2+1}$}

\noindent for all $\delta_1,\delta_2\in\nc$ and $u,v\in\nc/\nz$. Therefore the Poisson structure
induces two families of $\bW$-equivariant bilinear maps:

\noindent\centerline{
$P^{\delta_1,\delta_2}_{u,v}:\Omega_u^{\delta_1}\times \Omega_v^{\delta_2}\rightarrow
\Omega_{u+v}^{\delta_1+\delta_2}, (m,n)\mapsto m.n$,  }

\noindent\centerline{
$B^{\delta_1,\delta_2}_{u,v}:\Omega_u^{\delta_1}\times \Omega_v^{\delta_2}
\rightarrow \Omega_{u+v}^{\delta_1+\delta_2+1},
(m,n)\mapsto \{m,n\}$.}

\noindent It is clear that all these maps are non-degenerate, except
$B^{0,0}_{u,v}$. Indeed we have  $B^{0,0}_{u,v}=0$, for all $u,v\in\nc/\nz$.

\subsection{The extended Lie algebra $\cP_{\xi}$}

Recall Kac's construction of an extended Lie algebra \cite{Kac}. 
Start with a triple $(\fg,\kappa,\delta)$, where 
$\fg$ is a Lie algebra, $\delta:\fg\rightarrow\fg, x\mapsto\delta.x$ is a
derivation and $\kappa:\fg\times\fg\rightarrow\nc$ is a symmetric
$\fg$-equivariant and $\delta$-equivariant bilinear form.
Then the extended Lie algebra  is the vector space
$\fg_{e}=\fg\oplus \nc\,\delta\oplus \nc\,c$ and its Lie bracket 
$[,]_e$ is defined by the following relations:

$[x,y]_e=[x,y]+\kappa(x,\delta.y)c$,

$[\delta,x]_e=\delta.x$,

$[c,\fg_e]_e=0$,

\noindent for any $x,y\in\fg$ and where $[x,y]$ is the Lie bracket in
$\fg$.

We will apply this construction to the Lie algebra ${\cP}$.
The residue map $\Res:{\cal P}\rightarrow {\nc}$ is defined as follows:
$\Res(\Omega_u^{\delta})=0$ for $(\delta,u)\neq (1,0)$ and the restriction
of $\Res$ to $\Omega^1_0$ is the usual residue. Thus
set $\kappa(x,y)=\Res\,xy$ for any $x$, $y\in\cP$.
Let $(a,b)$ be projective coordinates of $\xi\in\np^1$. Define the derivation 
$\delta_\xi$ of $\cP$ by
$\delta_\xi x=z^{b-a}\partial^{-a}\{z^{a-b}\partial^{a},x\}$ for
any $x\in\cP$. Informally, we have 
$\delta_\xi x=\{\log( z^{a-b}\partial^{a}),x\}$.

It is easy to check that $\kappa$ is equivariant under
$\ad(\cP)$ and under $\delta_\xi$.  The two-cocycle
$x,y\in\cP\mapsto \Res(x\delta_\xi y)$ is 
the Khesin-Kravchenko cocycle \cite {KK}.
The corresponding extended Lie
algebra will be denoted by ${\cP}_\xi$. Thus we have

\noindent\centerline{
${\cP}_\xi={\cP}\oplus\nc\, \delta_\xi\oplus\nc\, c$.}

Since ${\cP}_\xi$ is not a Poisson algebra, its Lie bracket will be denoted by $[\,,]$.
Set ${\cP}^+_\xi=[{\cP}_\xi,{\cP}_\xi]$ and
${\cP}^{-}_\xi={\cP}_\xi/Z({\cP}_\xi)$, where $Z({\cP}_\xi)$ is the center
of ${\cP}_\xi$. As before $\bW$ is a Lie subalgebra of
${\cP}_\xi$, and the $\bW$-modules
${\cP}^{\pm}_\xi$ decomposes as follows

\noindent\centerline{
$\cP^{\pm}_\xi=\oplus_{(\delta,u)\in \nc\times\nc/\nz}\,\,\Omega_u^{\delta}(\xi,\pm)$,
}

\noindent where $\Omega_u^{\delta}(\xi,\pm)=\Omega_u^{\delta}$ or all 
$(\delta,u)\in \nc\times\nc/\nz$ except that
$\Omega_0^{0}(\xi,-)\simeq A_\xi$ and $\Omega_0^{1}(\xi,+)\simeq B_\xi$.
The Lie bracket of $\cP_\xi$ induces a bilinear map

\noindent\centerline{
$B(\xi):\cP^{-}_\xi\times \cP^{-}_\xi\rightarrow \cP^{+}_\xi,
(m\, \hbox{\rom modulo}\, Z(\cP_\xi),n\, \hbox{\rom modulo}\,Z(\cP_\xi)\mapsto [m,n]$.
}
 
\noindent As before,
the components of $B(\xi)$ provide the following  $\bW$-equivariant bilinear maps

\noindent\centerline{
$B^{\delta_1,\delta_2}_{u,v}(\xi):\Omega_u^{\delta_1}(\xi,-)\times 
\Omega_v^{\delta_2}(\xi,-)
\rightarrow \Omega_{u+v}^{\delta_1+\delta_2+1}(\xi,+),
(m,n)\mapsto [m,n]$.}

\noindent It should be noted that if $\delta_1\delta_2(\delta_1+\delta_2)\neq 0$,
then we have  $B^{\delta_1,\delta_2}_{u,v}(\xi)=B^{\delta_1,\delta_2}_{u,v}$.

\subsection{Other $\bW$-equivariant bilinear maps.}

{\it The Grozman operator:} Among the $\bW$-equivariant bilinear maps between modules of the class $\cS$, the most surprizing is the Grozman 
operator. It  is the bilinear map $G_{u,v}:\Omega_{u}^{-2/3}\times
\Omega_{v}^{-2/3}\rightarrow \Omega_{u+v}^{5/3}$, defined by the following formula

\noindent
\centerline{$G_{u,v}(e^{-2/3}_x,e^{-2/3}_y)=(x-y)(2x+y)(x+2y)e^{5/3}_{x+y}$.}

\smallskip
\noindent
{\it The bilinear map $\Theta_{\infty}$:} 
Let $\xi\in\np^1$ with projective
coordinates $(a,b)$.
 Define $\Theta_{\xi}: A_{a,b}\times A_{a,b}\rightarrow
B_{a,b}$ by the following requirements:

$\Theta_{\xi}(u^A_m,u^A_n)=0$ if $mn(m+n)\neq 0$ or if $m=n=0$

$\Theta_{\xi}(u^A_0,u^A_m)=-\Theta(u^A_m,u^A_0)=1/m\,u^B_m$ if $m\neq 0$,

$\Theta_{\xi}(u^A_{-m},u^A_m)=1/m\,u^B_0$ if $m\neq 0$.

It is easy to see  that $a\Theta_{\xi}$ is identical to the bracket
$-B^{0,0}_{0,0}(a,b)$. In particular, $\Theta_{\xi}$ is $\bW$-equivariant
(for $a=0$, this follows by extension of polynomial identities). So
$\Theta_{\infty}$ is the only  new bilinear map, since, for $\xi\neq \infty$, 
$\Theta_\xi$ is essentially the bracket of $\cP_\xi$.

\smallskip
\noindent
{\it The bilinear map $\eta(\xi_1,\xi_2,\xi_3)$:}
Let $\xi_1,\xi_2,\xi_3$ be  points in $\np^1$ which are not all equal, with projective
coordinates
$(a_1,b_1),\,(a_2,b_2),\,(a_3,b_3)$. Choose a non-zero triple $(x,y,z)$ such that 
$z(a_3,b_3)=x(a_1,b_1)+y(a_2,b_2)$. Recall that all $A_\xi$ have the same underlying vector
space, therefore the map

\noindent\centerline{$y\Res\times id +x id\times \Res: A_{a_1,b_1}\times A_{a_2,b_2}
\rightarrow A_{a_3,b_3}$}

\noindent is well defined.  This defines a map, up to a scalar multiple,

\noindent\centerline{
$\eta(\xi_1,\xi_2,\xi_3):A_{\xi_1}\times A_{\xi_2}\rightarrow A_{\xi_3}$,}

\noindent which is clearly $\bW$-equivariant.

\smallskip 
\noindent {\it The obvious map $P^M$:}
 Also for each $M\in\cS$ denote by 
$P^M$ the obvious map
$((a+x),m)\in (\overline A\oplus\nc)\times M\mapsto xm\in M$.

\subsection{ Primitive bilinear maps}

Let $\cN$ be the class of all non-zero $\bW$-equivariant maps $\phi:M\rightarrow N$,
where $M,N$ are non-isomorphic modules of the class $\cS$. Up to conjugacy, there are only two possibilities:

(i) $M$ is in the $A$-family,  $N$ is in the $B$-family
and $\phi$ is the morphism $M\twoheadrightarrow\nc\hookrightarrow N$, or

(ii) $M$ is in the $B$-family,   $N$ is in the $A$-family and
$\phi$ is the morphism $M\twoheadrightarrow\overline A\hookrightarrow N$.

\noindent Hence, the condition $M \not\simeq N$ means that $M$ and $N$ are not simultaneously
isomorphic to $\overline A\oplus \nc$. 

Let $M,M',N,N',P$ and $P'$ in $\cS$. Let 
$\phi:M'\rightarrow M$, $\psi:N'\rightarrow N$ and 
$\theta:P\rightarrow P'$ be $\bW$-equivariant maps,
and let $\pi\in \B_{\bW}(M\times N, P)$ be a $\bW$-equivariant bilinear map.
Set $\pi'=\theta\circ\pi\circ(\phi\times\psi)$. If at least one of the three morphisms
$\phi$, $\psi$ or $\theta$ is of the class $\cN$, then $\pi'$ is called an 
{\it imprimitive form} of $\pi$.

A $\bW$-equivariant bilinear map between modules of the class $\cS$
is called {\it primitive} if is  not a linear combinations of imprimitive forms.
The composition of any three composable morphisms of the class $\cN$ is  zero. It follows
easily that any $\bW$-equivariant bilinear maps between modules of the class $\cS$
is either primitive, or it is a linear combination of imprimitive bilinear form. 
Thus the classification of all $\bW$-equivariant bilinear maps between 
modules of the class $\cS$ reduces to the classification of primitive ones.

\begin{lemma} \label {primitivity}
Let $M,N$ and $P$ be $\bW$ modules of the class $\cS$, and let
$\pi\in\B_{\bW}(M\times N,P)$. Assume one of the following conditions holds

(i) $M$ and $N$ are irreducible and $P$ is the linear span of $\pi(M\times N)$

(ii) $N$ and $P$ are irreducible and the left kernel of
$\pi$ is zero.

Then   $\pi$ is primitive.
\end{lemma}

This obvious  lemma is  useful to check easily that some 
bilinear maps are primitive. Some of the bilinear forms defined in this section are not
primitive. It will be proved in Corollary \ref{cor1} of Section \ref{sect_conclusion}
that, up to $\fS_3$-symmetry,   all primitive bilinear forms 
between modules of the class $\cS$
have been defined in this section.

\subsection{Examples of $\bW$-equivariant germs}\label{sect_example-germ}

In what follows, the elements of $\nc^3$ will be written as triples
$(\delta_1,\delta_2,\gamma)$. Let $\sigma$ be the involution
defined by 
$(\delta_1,\delta_2,\gamma)^\sigma=(\delta_2,\delta_1,\gamma)$.
Let $\fz$ be the union 
of the two affine planes $H_i$, the six affine lines $D_i$ and the five
points $P_i$ defined as follows.
For $i=0,1$, the plane $H_i$ is  defined by the equation $\gamma=\delta_1+\delta_2+i$.
The six lines $D_i$ are parametrized as follows:

$D_1=\{(0,\delta,\delta+2)\vert \delta\in\nc\}$ and $D_2= D_1^\sigma$,

$D_3=\{(\delta,1,\delta)\vert \delta\in\nc\}$ and $D_4=D_3^\sigma$,

$D_5=\{(\delta,-(1+\delta),1)\vert \delta\in\nc\}$ 

$D_6=\{(\delta,1-\delta,0)\vert \delta\in\nc\}$.

\noindent The fives points are 
$P_1=(0,0,3)$, $P_2=(0,-2,1)$,
$P_3=P_2^\sigma$ and $P_4=(1,1,0)$ and
$P_5=(-2/3,-2/3,5/3)$.  Also
let $\fz^\ast$ be the set of all $(\delta_1,\delta_2,\gamma)\in\fz$ such that
$\{\delta_1,\delta_2,\gamma\}\not\subset\{0,1\}$. 

In what follows, we will consider $\Omega_u^\delta$ as a module of
the class ${\cS}^\ast$. Thus  we can define without
ambiguity the   degree of any 
$\pi\in\cG_{\bW}(\Omega_u^{\delta_1}\times\Omega_v^{\delta_2},
\Omega^\gamma_{u+v})$  as the scalar $\gamma-\delta_1-\delta_2$.
The following  table provide a list of germs $\pi$
in $\cG_{\bW}(\Omega_u^{\delta_1}\times\Omega_v^{\delta_2},
\Omega^\gamma_{u+v})$, when $(\delta_1,\delta_2,\gamma)$ runs over $\fz^\ast$.
In the table, we omit the symbol $\cG$. For example, $d^{-1}$ stands for $\cG(d)^{-1}$, which is well-defined even for $u\equiv 0 \mod \nz$.

\begin{table}[h]
\caption{\textbf{List  of $\pi\in\cG_{\bW}(\Omega_u^{\delta_1}\times\Omega_v^{\delta_2},
\Omega^\gamma_{u+v})$,
where $(\delta_1,\delta_2,\gamma)$ runs over $\fz^*$}}
\label{table1}

\begin{center}
\begin{tabular}{|c|c|c||c|} \hline

&$\deg \pi$ & $(\delta_1,\delta_2,\gamma)$&$\pi$ \\ \hline

1. & $3$ &  $(-\frac{2}{3}, -\frac{2}{3}, \frac{5}{3})$ & $G_{u,v}$\\ \hline

2.& $3$ & $(0,0,3)$ & $B^{1,1}_{u,v}\circ(d\times d)$\\ \hline

3.& $3$ & $(0,-2,1)$ & $d\circ B^{1,-2}_{u,v}\circ (d\times id)$ \\ \hline

4.& $3$ & $(-2,0,1)$ & $d\circ B^{-2,1}_{u,v}\circ (id\times d)$ \\ \hline

5.& $2$ & $(0,\delta, \delta+2)$ & $B^{1,\delta}_{u,v}\circ(d\times id)$ \\ \hline

6.& $2$ & $(\delta,0, \delta+2)$ & $B^{\delta,1}_{u,v}\circ(id\times d)$ \\ \hline

7.& $2$ & $(\delta, -\delta-1,1)$ & $d\circ B^{\delta,-\delta-1}_{u,v}$\\ \hline

8.& $1$ & $(\delta_1,\delta_2, \delta_1+\delta_2+1)$ & $B^{\delta_1,\delta_2}_{u,v}$ \\ \hline

9.& $0$ & $(\delta_1,\delta_2, \delta_1+\delta_2)$ & $P^{\delta_1,\delta_2}_{u,v}$ \\
\hline

10.& $-1$ & $(1,\delta,\delta)$ & $P^{0,\delta}_{u,v}\circ (d^{-1}\times
id)$\\ \hline

11.& $-1$ & $(\delta,1,\delta)$ &$P^{\delta,0}_{u,v}\circ(id\times d^{-1})$\\ \hline

12.& $-1$ & $(\delta,1-\delta,0)$ & $d^{-1}\circ P^{\delta,1-\delta}_{u,v}$ \\ \hline
\end{tabular}
\end{center}
{\small The condition $(\delta_1,\delta_2,\gamma)\in\fz^*$ implies that
$(\delta_1,\delta_2)\neq (0,0)$ in the line 8, $(\delta_1,\delta_2)\neq (0,0), (0,1)$
or $(1,0)$ in the line 9, $\delta\neq 0$ or $1$ in the lines 10-12.}
\end{table}

Let $M,\,N$ and $P$ be $\bW$-modules of the class $\cS$. Set $u=\Supp\,M$,
$v=\Supp\,N$ and assume that $\Supp \,P=u+v$. Let 
$\delta_1\in\deg M$,
$\delta_2\in\deg N$ and $\gamma\in\deg P$. It follows from Lemma \ref{W-germs}
that 

\noindent
\centerline{$\cG(M)=\cG(\Omega_u^{\delta_1})$,
$\cG(N)=\cG(\Omega_v^{\delta_2})$ and
$\cG(P)=\cG(\Omega_{u+v}^\gamma)$.}

\begin{lemma}\label{lower}
Assume that $(\delta_1,\delta_2,\gamma)\in\fz$. Then 
$\cG_\bW(M\times N,P)$ is not zero, and moreover
we have $\dim\,\cG_\bW(M\times N,P)\geq 2$ if 
$\{\delta_1,\delta_2,\gamma\}\subset\{0,1\}$.
More precisely we have:

(i) for $(\delta_1,\delta_2,\gamma)\in\fz^*$,  Table \ref{table1}
provides a non-zero $\pi\in\cG_\bW(M\times N,P)$,

(ii) if $\{\delta_1,\delta_2,\gamma\}\subset\{0,1\}$, then we have
$\cG(M)=\cG(\Omega_u^{0})$,
$\cG(N)=\cG(\Omega_v^{0})$ and
$\cG(P)=\cG(\Omega_{u+v}^1)$ and the  maps
$\pi_1:=P^{0,1}_{u,v}\circ(id\times d)$ and $\pi_2:=P^{1,0}_{u,v}\circ(d\times id)$
are non-proportional  elements of $\cG_\bW(M\times N,P)$.

\end{lemma}

Theorem 2, proved in  Section 7,  states that the maps listed in the previous lemma
provide  a basis of $\cG_{\bW}(M\times N,P)$. 
It also states that
$\cG_\bW(M\times N,P)=0$ if $(\delta_1,\delta_2,\gamma)\not\in\fz$.

\section{Classification of $\bW$-equivariant degenerate bilinear maps}\label{sect_Thm1}
Let $M, N$ and $P$ be in the class $\cS$. The goal of the section is the
classification of
all $\bW$-equivariant degenerate bilinear maps
$\pi:M\times N\rightarrow P$. 
In order to simplify the statements, we will always assume that 
$M, N$ and $P$ are indecomposable. 

Assume that $\pi\neq 0$ and set $S=\overline{\Supp\, \pi}$.
By Lemma \ref{lemma_support}, $S\subset H\cup D\cup V$ is an union of lines. Thus there are
four cases, of increasing complexity:

\begin{enumerate}
\item[(i)] $S=0$,
\item[(ii)] $S$ consists of one line $H,\, D$ or $V$,
\item[(iii)] $S$ consists of two lines among $H,\, D$ and $V$, or
\item[(iv)] $S=H\cup D\cup V$.
\end{enumerate}

Since the three lines $H$, $V$ and $D$ are exchanged by the
$\fS_3$-symmetry, we can reduce the full classification to
the following four cases:

\begin{enumerate}
\item[(i)]$S=0$, 
\item[(ii)] $S=V$, 
\item[(iii)] $S=H\cup V$, 
\item[(iv)]$S=H\cup V\cup D$. 
\end{enumerate}

In the following lemmas, it is assumed that $\pi:M\times N\rightarrow P$ is a 
non-zero degenerate $\bW$-equivariant  bilinear map.

\begin{lemma} If $S=0$, then $\pi$ is trivial.
\end{lemma}

The lemma is obvious. 

A $\bW$-equivariant map $\psi: N\rightarrow P$ is
called an {\it   almost-isomorphism} if its kernel has dimension $\leq
1$. The only  almost-isomorphisms which are not isomorphisms between modules
of the class $\cS$ are the 
maps from $B_{\xi}$ to $A_{\eta}$ obtained as the composition of
$B_{\xi}\twoheadrightarrow\overline{A}$ and $\overline{A}\hookrightarrow A_{\eta}$.

\begin{lemma}\label{lemma_deg-V}
Assume that $S=V$. Then there is
surjective maps $\phi:M\rightarrow \nc$ and
almost-isomorphism 
$\psi: N\rightarrow P$ such that

\centerline{$\pi(x,y)=\phi(x)\psi(y)$,}

\noindent for any $(x,y)\in M\times N$.
\end{lemma}

Lemma \ref{lemma_deg-V} is obvious.

In order to investigate the case where $S$ contains two lines, 
it is necessary to state a diagram chasing lemma.
Let $\fg$ be a Lie algebra, let $X$ be a $\fg$-module. For
$\xi\in H^1(\fg,M)$,  denote by $X_\xi$ the corresponding extension:

\centerline{$0\rightarrow X\rightarrow X_\xi\rightarrow \nc
\rightarrow 0$.}

\begin{lemma} \label{chasing}
Assume that $End_{\fg}(X)=\nc$ and that $X=\fg.X$.
Let  $\xi_1,\xi_2,\xi_3\in H^1(\fg,X)$ and set
$Z=X_{\xi_1}\otimes X_{\xi_2}/X\otimes X$. We have

\centerline{$\dim
\Hom_{\fg}(Z,X_{\xi_3})=3-r$,}

\centerline{$\dim
\Hom_{\fg}(Z,X)=2-s$,}

\noindent where $r$ is the rank of $\{\xi_1,\xi_2,\xi_3\}$
and $s$ is the rank of  $\{\xi_1,\xi_2\}$ in $H^1(\fg,M)$.
\end{lemma}

\begin{proof}
For $i=1$ to $3$, there are elements $\delta_i\in X_{\xi_i}\setminus X$   such that
the map $x\in\fg\rightarrow x.\delta_i\in X$ is a cocycle of the class $\xi_i$.
Since $X_{\xi_i}=X\oplus\nc \delta_i$, we have
 
\noindent\centerline{$Z= [\delta_1\otimes X]\oplus [X\otimes\delta_2] \oplus
\nc(\delta_1\otimes\delta_2)$.}

Let  $\pi\in \Hom_{\fg}(Z,X_{\xi_3})$. 
Note that $\delta_1\otimes X$ is a submodule of $Z$ isomorphic with $X$.
Since $\fg.X=X$, we have $\pi(\delta_1\otimes X)\subset X$. Thus 
there exists $\lambda\in\nc$ such that 
$\pi(\delta_1\otimes x)=\lambda x$ for all $x\in X$. Similarly,
there exists $\mu\in\nc$ such that 
$\pi(x\otimes\delta_2)=\mu x$ for all $x\in X$. By definition,
we have $\pi(\delta_1\otimes\delta_2)=\nu\delta_3 +x_0$ for some
$\nu\in\nc$ and some $x_0$ in $X$. 

The $\fg$-equivariance of $\pi$ is equivalent to
the equation:

\noindent\centerline{ $\mu g.\delta_1+\lambda g.\delta_2
=\nu g.\delta_3  +g.x_0$ for all $g\in \fg$.}

\noindent Thus
$\dim\Hom_{\fg}(Z,X_{\xi_3})$ is
exactly the dimension of the space of triples $(\lambda,\mu,\nu)\in \nc^3$
such that
 
\centerline{$\lambda\delta_2 +\mu \delta_1  -\nu\delta_3\equiv 0$ in
$H^1(\fg,X)$,}

\noindent and the first assertion  follows. The second assertion is
similar.
\end{proof}

For  $\xi\in\np^1$ and $t\in\nc$,
define $\eta^{t}_{(\xi)}: A_{\xi}\times
A_{\xi}\rightarrow A_{\xi}$ by the formula 

\noindent\centerline{$\eta^{t}_{\xi}(m,n)=  \Res(m) n+ t \Res(n) m$,}

\noindent and recall that $\eta(\xi_1,\xi_2,\xi_3)$ is defined in Section $4.3$.

\begin{lemma}
Assume that $S=V\cup H$. 
Then  $\pi$ is conjugate to one  of the following:

(i) $\eta(\xi_1,\xi_2,\xi_3)$, for some $\xi_1,\xi_2,\xi_3\in\np^1$ with
$\xi_3\notin\{\xi_1,\xi_2\}$, or

(ii) $\eta^{t}_{\xi}$ for some $t\neq 0$ and $\xi\in\np^1$.
\end{lemma}

\begin{proof}
By Lemma \ref{lemma_support} we have 
$\pi(M_\pi\times N_\pi)=0$, $M/M_\pi=\nc$ and 
$N/N_\pi=\nc$. It follows that
$M\simeq A_{\xi_1}$ and $N\simeq A_{\xi_2}$ for some
$\xi_1,\xi_2\in\np^1$.
Thus $(M/M_{\pi})\otimes N_{\pi}$ is isomorphic to
$\overline {A}$. Since  $\pi$
induces a non-zero map
$(M/M_{\pi})\otimes N_{\pi}\rightarrow P$, the $\bW$-module
$P$ contains $\overline{A}$, hence $P$ is isomorphic
to $A_{\xi_3}$ for some non-zero $\xi_3\in\np^1$.

Let $B=\{\mu\in \B_{\bW}(M\times N,P)\vert \mu(M_\pi\times N_\pi)=0\}$. It follows from
the Kaplansky-Santharoubane Theorem that
$\dim H^1(W,\overline A)=2$.
Thus if  $\xi_1,\xi_2,\xi_3$
are not all equal, it follows from Lemma \ref{chasing}
that $B=\nc\, \eta(\xi_1,\xi_2,\xi_3)$. However $\Supp\,\eta(\xi_1,\xi_2,\xi_3)$
lies inside $H$ or $V$ if $\xi_1=\xi_3$ or $\xi_2=\xi_3$. Hence we have
$\xi_3\notin\{\xi_1,\xi_2\}$
and  $\pi$ is conjugate to
$\eta(\xi_1,\xi_2,\xi_3)$. Similarly, if all $\xi_i$ are equal to some
$\xi\in \np^1$, then
$B$ is the two dimensional vector space generated by the affine line
$\{\eta^{t}_{\xi}\vert\,t\in\nc\}$. Thus $\pi$ is conjugate to some
$\eta^t_\xi$. Moreover the hypothesis $\Supp\,\pi= H\cup V$ implies
that $t\neq 0$.
\end{proof}

\begin{lemma}\label{lemma_class-deg2}
Assume that $S=V \cup H \cup D$.
Then, $\pi$ is conjugate to $\Theta_{\xi}$ for some $\xi\in \np^1$, 
modulo a trivial map.
\end{lemma}

\begin{proof}
It follows from Lemma \ref{lemma_support} that we have 
$\pi(M_\pi \times N_\pi) \cong \nc, M/M_\pi \cong \nc$ and $N/N_\pi \cong
\nc$. Thus $M=A_{\xi_1}, N=A_{\xi_2}$ and $P=B_{\xi_3}$ form some
 $\xi_i \in \np^1$.

Set $Z=A_{\xi}\otimes A_{\xi}/\overline{A}\otimes \overline{A}$.
By Lemma \ref{lemma_support}, the composition of
any $\mu\in \B^0_{\bW}(A_{\xi_1}\times A_{\xi_2}, B_{\xi_3})$
with the map $B_{\xi_3}\rightarrow\overline{A}$
provides a linear map  ${\overline\mu}: Z\rightarrow \overline A$.
Since $\overline{\mu}$ is not zero, 
Lemma \ref{chasing} implies that $\xi_1=\xi_2$.  By the
$\fS_3$-symmetry,  $B_{\xi_3}$ is the restricted dual of $A_{\xi_2}$, so we have
$\xi_3=\xi_1$. Set $\xi:=\xi_1=\xi_2=\xi_3$ 
and  consider the following exact sequence

\noindent\centerline
{$0 \rightarrow \B_{\bW}(A_\xi\times A_\xi, \nc) \rightarrow
\B_{\bW}^0(A_\xi\times A_\xi, B_{\xi}) 
\rightarrow \Hom_{\bW}(Z,B_{\xi}/\nc)$,}

\noindent where the last arrow is the map $\mu\mapsto\overline{\mu}$.

It is clear that the subspace $\B_{\bW}(A_{\xi}\times A_{\xi}, \nc)$
of  $\B_{\bW}^{0}(A_{\xi}\times A_{\xi}, B_{\xi})$ is 
the space of trivial bilinear maps. By the previous lemma, 
$\Hom_{\bW}(Z,B_{\xi}/\nc)$ has dimension one. It follows from its definition that
$\Supp \Theta_\xi=H\cup V\cup D$. Hence we have

\noindent\centerline{$\B_{\bW}^{0}(A_{\xi}\times A_{\xi}, B_{\xi})=
\B_{\bW}(A_{\xi}\times A_{\xi}, \nc)\oplus \nc \Theta_\xi$.}

\noindent Hence $\pi$ is conjugate to 
$\Theta_{\xi}$ modulo a trivial map.
\end{proof}

\begin{thm} 
 Let $M,N$ and $P$ indecomposable modules of the class $\cS$
and let $\pi:M\times N\rightarrow P$ be a $\bW$-equivariant degenerate bilinear map.
Up to the $\fS_3$-symmetry, $\pi$ is conjugate to one of the following:

(i) a trivial bilinear map $\pi:A_{\xi_1}\times A_{\xi_2}\rightarrow B_{\xi_3}$,

(ii) the map $\pi:A_{\xi}\times N\rightarrow P$,
$(m,n)\mapsto \Res(m)\psi(n)$ where 
$\xi\in\np^1$ and where $\psi:N\rightarrow P$ is an
almost-isomorphism,

(iii) the map $\eta(\xi_1,\xi_2,\xi_3)$ for some 
$\xi_1,\xi_2$ and $\xi_3\in  \np^1$  with
$\xi_3\notin\{\xi_1,\xi_2\}$,
 or the map
$\eta^{t}_{\xi}$ for some
$\xi\in \np^1$  and $t\neq 0$,

(iv) $\Theta_{\xi}+\tau$, where $\xi\in\np^1$ and $\tau$ is a trivial map.

\end{thm}

The following table is another presentation of Theorem 1. To limit
the number of cases, the list is given up to the $\fS_3$-symmetry.
That is why the datum in the third column is $P^*$ and not $P$.

\begin{table}[h]
\caption{\textbf{List, up to the $\fS_3$-symmetry, of possible $S$ and $d$,
where $d=\dim \B_{\bW}^0(M\times N, P)$ and
$S=\overline{\Supp\,\pi}$ for  
$\pi\in \B_{\bW}^0(M\times N, P)$.}}
\label{table2}

\begin{center}

\begin{tabular}{|c|c|c|c|c|c|} \hline
&$M\times N$ & $P^*$& $S$&d
& Restrictions \\ \hline

1.& $A_{\xi_1}\times A_{\xi_2}$ & $A_{\xi_3}$ & $0$& $1$
&$Card \{\xi_1,\xi_2,\xi_3\}\geq 2$\\ \hline

2.& $A_{\xi}\times X$ & $X^*$ & $V$& $1$
&$X\not\simeq A_\xi$ or $B_\xi$
\\ \hline

3.& $A_{\xi_1}\times B_{\xi_2}$ & $B_{\xi_3}$ & $V$& $1$&
\\ \hline

4.& $A_{\xi_1}\times A_{\xi_2}$ & $B_{\xi_3}$ & $H\cup V$
& $1$ &$\xi_3\notin\{\xi_1,\xi_2\}$
\\ \hline

5.& $A_{\xi}\times A_{\xi}$ & $B_{\xi}$& 
$H$, $V$ or  $H\cup V$&$2$ &
\\ \hline

6.& $A_{\xi}\times A_{\xi}$ & $A_{\xi}$& 
$0$ or  $H\cup V\cup D$&$2$ &
\\ \hline

\end{tabular}
\end{center}

\small{ Except the indicated restrictions, $X\in\cS$ is  arbitrary 
and $\xi,\xi_1,\xi_2$ and $\xi_3$ are arbitrary.}
\end{table}

\section{Bounds for the dimension of the spaces of germs of bilinear maps}
Let  $M,N$ and $P$ be $\bW$-modules of the class $\cS$. 
In this section we introduce a space 
$\tilde \cG_{\fsl(2)}(M\times N, P)$ which is a
good approximation of $\cG_{\bW}(M\times N, P)$. 
Indeed we have:

\noindent\centerline
{$\cG_{\bW}(M\times N\, P)\subset
\tilde \cG_{\fsl(2)}(M\times N\, P)
\subset \cG_{\fsl(2)}(M\times N, P)$.}

\subsection{On $\dim \cG_{\fsl(2)}(M\times N,P)$}

For an element $\gamma\in\nc^2$, set 
$L.\gamma=\gamma+(1,0)$ and $R.\gamma=\gamma+(0,1)$. Similarly,
for a pair $\{\alpha,\beta\}$ of elements of
$\nc^2$, set $L.\{\alpha,\beta\}=\{L.\alpha,L.\beta\}$ and
$R.\{\alpha,\beta\}=\{R.\alpha,R.\beta\}$. 
The pair $\{\alpha,\beta\}$  is called {\it adjacent} if
$\beta=\alpha+(1,-1)$ or $\alpha=\beta+(1,-1)$. 
Thus $L.\{\alpha,\beta\}$ and 
$R.\{\alpha,\beta\}$ are adjacent whenever $\{\alpha,\beta\}$ is adjacent.
Given $(x,y)\in\nc^2$, let 
$C(x,y)$ be the set of all elements of $\nc^2$ of the form
$(x+m, y+n)$ with $m,n\in\nz_{\geq 0}$ and
$(m,n)\neq (0,0)$.

Let $\delta_1,\,\delta_2$ and $\gamma$ be scalar, and let
$u$ and  $v$ be $\nz$-cosets. Let 
$\pi:\Omega^{\delta_1}_{u}\times
\Omega^{\delta_2}_{v}\rightarrow
\Omega^{\gamma}_{u+v}$ be an $L_0$-equivariant bilinear map.
Let $(e_x^{\delta_1})_{x\in{u}}$ be the basis of
$\Omega^{\delta_1}_{u}$ defined in Section 1.
Similarly
denote by $(e_y^{\delta_2})_{y\in{v}}$ and
$(e_z^{\gamma})_{z\in{u+v}}$ the corresponding bases
of $\Omega^{\delta_2}_{v}$ and
$\Omega^{\gamma}_{u+v}$.

Since $\pi$ is $L_0$-equivariant, there exists a function
$X: u\times v\rightarrow\nc$ defined by the identity:

\noindent\centerline{$\pi(e_x^{\delta_1},e_y^{\delta_2})=X(x,y)
e_{x+y}^{\gamma}$.}

\begin{lemma}\label{tech} 
Assume there exists $(x,y)\in u\times v$ such that:

(i) $\Supp (L_{\pm1}.\pi) \cap C(x,y)=\emptyset$,

(ii)  $\Re x>\pm\Re\delta_1$, 
$\Re y>\pm\Re\delta_2$, $\Re (x+y)>\pm\Re\gamma$, and

(iii) $X(x+1,y)=X(x,y+1)=0$.

Then $\cG(\pi)=0$
\end{lemma}

\begin{proof}

{\it First step:} We claim that 
for any adjacent pair $\{\alpha,\beta\}$ in 
$C(x,y)$ with $X(\alpha)=X(\beta)=0$, then $X$ vanishes on 
$R^k.L^l.\{\alpha,\beta\}$ for any $k,l\in\nz_{\geq 0}$.

First prove that $X$ vanishes on $L. \{\alpha,\beta\}$
Indeed we can assume that $\alpha=\beta+(1,-1)$, and therefore we have
 $\alpha=(x'+1,y')$, $\beta=(x',y'+1)$ for some $(x',y')\in \nc^2$.

Since $L_{-1}.\pi(e_{x'+1}^{\delta_1},e_{y'+1}^{\delta_2})$
is a linear combination of $\pi(e_{x'}^{\delta_1},e_{y'+1}^{\delta_2})$
and 

\noindent $\pi(e_{x'+1}^{\delta_1},e_{y'}^{\delta_2})$, we get

\noindent\centerline{$0= L_{-1}.\pi(e_{x'+1}^{\delta_1},e_{y'+1}^{\delta_2})=
X(x'+1,y'+1)L_{-1}. e_{x'+y'+2}^{\gamma}$.}

\noindent  Since $\Re (x'+y'+2-\gamma)>0$, we get 
$L_{-1}. e_{x'+y'+2}^{\gamma}\neq 0$ and therefore $X(x'+1,y'+1)=0$.
Moreover we have

\noindent\centerline{$0=L_{1}.\pi(e_{x'+1}^{\delta_1},e_{y'}^{\delta_2})
=\pi(L_1.e_{x'+1}^{\delta_1},e_{y'}^{\delta_2})
+\pi(e_{x'+1}^{\delta_1},L_1.e_{y'}^{\delta_2})$.}

\noindent Using that $X(x'+1,y'+1)=0$, it follows
that $\pi(L_1.e_{x'+1}^{\delta_1},e_{y'}^{\delta_2})=0$.
Since $\Re(x'+1+\delta_1)>0$, we have 
$L_1.e_{x'+1}^{\delta_1}\neq 0$ and therefore 
$X(x'+2,y')=0$. Hence $X$ vanishes on $L.\{\alpha,\beta\}$.

Similarly, $X$ vanishes on $R.\{\alpha,\beta\}$. It follows by induction that
$X$ vanishes on  $R^k.L^l.\{\alpha,\beta\}$ for any $k,l\in\nz_{\geq 0}$.

{\it Second step:} Set $\alpha=(x+1,y)$ and $\beta=(x,y+1)$. 
We have $C(x,y)=\cup_{k,l\in\nz_{\geq 0}}  \,R^k.L^l.\{\alpha,\beta\}$.
It follows that $X$ vanishes on $C(x,y)$. Hence $\cG(\pi)=0$.

\end{proof}

Let 
$\pi\in \B_{L_0}(\Omega^{\delta_1}_{u}\times
\Omega^{\delta_2}_{v}, \Omega^{\gamma}_{u+v})$. As before, set 

\noindent\centerline{$\pi(e_x^{\delta_1},e_y^{\delta_2})=X(x,y)
e_{x+y}^{\gamma}$.}

\begin{lemma}\label{consecutive} 
Assume that $\cG(\pi)$ is $\fsl(2)$-equivariant and 
non-zero.  Then there exists $(x_0,y_0)\in u\times v$
such that 

\noindent\centerline{$X(\alpha)\neq 0$ or $X(\beta)\neq 0$,}

\noindent  for any adjacent pair
$\{\alpha,\beta\}$ in $C(x_0,y_0)$.
\end{lemma}

\begin{proof}
 Since  $\cG(\pi)$ is $\fsl(2)$-equivariant, we can choose 
$(x_0,y_0)\in u\times v$ such that:

(i) $\Supp L_{\pm1}\circ \pi\cap C(x_0,y_0)=\emptyset$

\noindent Moreover we can assume that $\Re x_0$ and 
$\Re y_0$ are big enough in order that:

(ii) $\Re x_0>\pm\Re\delta_1$, 
$\Re y_0>\pm\Re\delta_2$, $\Re (x_0+y_0)>\pm\Re\gamma$.

Let $\{(x+1,y),(x,y+1)\}$ be an adjacent
pair in $C(x_0,y_0)$. The couple $(x,y)$ satisfies the conditions
(i) and (ii) of the previous lemma. Since $\cG(\pi)\neq 0$, it
follows that $X(x+1,y)\neq 0$ or $X(x,y+1)\neq 0$. 

\end{proof}

Let $M,N$ and $P$ be $\bW$-modules of the class
$\cS$. 

\begin{lemma}\label{Gsl(2)}
We have $\dim\,\cG_{\fsl(2)}(M\times N,P)\leq 2$.
\end{lemma}

\begin{proof} 

Let  $\pi_1,\pi_2,\pi_3$ be any elements in $\B_{L_0}(M\times N,P)$ such that 
$\cG(\pi_i)\in  \cG_{\fsl(2)}(M\times N,P)$.

Since the space $\cG(M\times N,P)$ only depends on the
germs of $M$, $N$ and $P$, we can assume that
$M= \Omega^{\delta_1}_{u}$,
$N=\Omega^{\delta_2}_{v}$ and
$P=\Omega^{\gamma}_{w}$ for some
scalars $\delta_1,\delta_2$ and $\gamma$ and some $\nz$-cosets
$u$, $v$ and $w$.
Moreover, we can assume $w=
u+v$, since otherwise
it is obvious that $\cG_{\fsl(2)}(M\times N,P)=0$.

There is $(x,y)\in u\times v$ such that
$\Supp (L_{\pm1}.\pi_i)\cap C(x,y)=\emptyset$ for $i=1$ to $3$.
Adding some positive integers to $x$ and $y$ if necessary,
we can assume that 
$\Re x>\pm\Re\delta_1$, 
$\Re y>\pm\Re\delta_2$ and $\Re (x+y)>\pm\Re\gamma$.
There is a non-zero triple $(a,b,c)$ of scalars with:

\noindent\centerline{$[a\pi_1+b\pi_2+c\pi_3]
(e_{x+1}^{\delta_1},e_{y}^{\delta_2})=
[a\pi_1+b\pi_2+c\pi_3]
(e_{x}^{\delta_1},e_{y+1}^{\delta_2})=0$.}

\noindent It follows from Lemma \ref{tech} that 
$a\cG(\pi_1)+b\cG(\pi_2)+c\cG(\pi_3)=0$. Since any three
arbitrary elements $\cG(\pi_1),\cG(\pi_2)$ and $\cG(\pi_3)$ of
$\cG_{\fsl(2)}(M\times N,P)$ are linearly dependant, it follows that
$\dim\,\cG_{\fsl(2)}(M\times N,P)\leq 2$.

\end{proof}

\subsection{The recurrence relations}
Let $M,\,N$ and $P$ be three $\bW$-modules of the class
$\cS$. Let $\pi\in \B(M\times N,P)$. Set

$\tilde{\pi}(m,n)=L_{-2}L_2.\pi(m,n)-
\pi(L_{-2}L_2.m,n)-$

\hskip2cm$-\pi(m, L_{-2}L_2.n)-
\pi(L_{-2}.m, L_2.n)-
\pi(L_{2}.m, L_{-2}.n)$,

\noindent for all $(m, n)\in M\times N$. Note that  we have 

$\tilde{\pi}
=L_{-2}\circ(L_{2}.\pi)+(L_{-2}.\pi)\circ (L_{2}\times \id)+
(L_{-2}.\pi)\circ (\id \times L_{2})$. 

\noindent Similarly, for a germ  $\mu\in \cG(M\times N,P)$, set

\noindent\centerline{$\tilde{\mu} =L_{-2}\circ(L_{2}.\mu)
+(L_{-2}.\mu)\circ( L_{2}\times \id) + (L_{-2}.\mu)\circ (\id\times L_{2})$.}

\noindent Set 
$\tilde{\cG}_{{\fsl}(2)}(M\times N,P)=\{
\mu\in {\cG}_{{\fsl}(2)}(M\times N,P)\vert \,\tilde{\mu}=0\}$. 
It follows from the definitions that we have

\noindent
\centerline{
${\cG}_{{\bW}}(M\times N,P)\subset
\tilde{\cG}_{{\fsl}(2)}(M\times N,P)\subset
{\cG}_{{\fsl}(2)}(M\times N,P)$.
}

Let $\delta_1,\delta_2$ and $\gamma$ be three scalars. For $k=1,\,2$, set

\begin{align*}
&a_k(x,y)=(x+k\delta_1)(y-k\delta_2), \\
&b_k(x,y)=k^2(\delta_1+\delta_2-\gamma-\delta_1^2-\delta_2^2+\gamma^2)-2xy, \\
&c_k(x,y)=(x-k\delta_1)(y+k\delta_2). 
\end{align*}

\noindent Let $u$ and $v$ be two $\nz$-cosets and let $\pi:\Omega^{\delta_1}_{u}\times
\Omega^{\delta_2}_{v}\rightarrow
\Omega^{\gamma}_{u+v}$ be a $L_0$-equivariant bilinear map.
As before, define the function $X$ by the identity:

\noindent\centerline{$\pi(e_x^{\delta_1},e_y^{\delta_2})=X(x,y)
e_{x+y}^{\gamma}$.}

\begin{lemma}\label{Casimir}
Assume that the $\cG(\pi)$ belongs
to $\tilde{\cG}_{{\fsl}(2)}
(\Omega^{\delta_1}_{u}\times \Omega^{\delta_2}_{v},
\Omega^{\gamma}_{u+v})$. Then
there exists $(x_0,y_0)\in u\times v$ such that: 

\noindent\centerline{
$a_k(x,y)X(x+k,y-k)+b_k(x,y)X(x,y)+c_k(x,y)X(x-k,y+k)=0$,}

\noindent for $k=1,2$ and  all $(x,y)\in C(x_0,y_0)$.
\end {lemma}

\begin{proof}
For $k=1,2$, set
$\pi_k=L_{-k}\circ (L_k. \pi)+
(L_{-k}.\pi)\circ (L_{k}\times id)+ (L_{-k}.\pi)\circ (id\times L_k)$.
Since the germ of $\pi$ is $\fsl(2)$-equivariant we have $\cG(\pi_1)=0$
and since $\pi_2=\tilde\pi$ we also have $\cG(\pi_2)=0$.
Therefore there exist 
$(x_0,y_0)\in u\times v$ such that $\Supp \pi_1$ and $\Supp \pi_2$ do
not intersect $C(x_0,y_0)$. This condition is equivalent to the recurrence
relations of the lemma.
\end{proof}

\subsection{The case $\dim \tilde{\cG}_{{\fsl}(2)}(M \times N,P)=2$}

\begin{lemma}\label{dim_tilde_cG} 
Let $M,N$ and $P$ be $\bW$-modules of
the class $\cS$. If

\noindent\centerline
{$\dim \tilde{\cG}_{{\fsl}(2)}(M \times N,P)=2$,}

\noindent then one of the following assertions holds:

(i) $\deg M=\deg N= \deg P =\{0,1\}$,

(ii) $\deg M=-1/2$, $\deg N =\{0,1\}$ and $\deg P\in\{-1/2,3/2\}$

(iii) $\deg M =\{0,1\}$, $\deg N=-1/2$, and $\deg P\in\{-1/2,3/2\}$.
\end{lemma}

\begin{proof} 
Assume that $\dim \tilde{\cG}_{{\fsl}(2)}(M \times N,P)=2$.
We can assume that $M=\Omega^{\delta_1}_{u}, \,
N=\Omega^{\delta_2}_{v}$ and
$P=\Omega^{\gamma}_{u+v}$ for some 
$\delta_1,\,\delta_2$ and $\gamma$ in $\nc$ and
some $u,\,v\in {\nc/\nz}$.
Choose $\pi_1,\pi_2\in\B_{L_0}(\Omega^{\delta_1}_{u} 
\times\Omega^{\delta_2}_{v},
\Omega^{\gamma}_{u+v})$  whose germs form a basis of 
$\tilde{\cG}_{{\fsl}(2)}
(\Omega^{\delta_1}_{u} 
\times\Omega^{\delta_2}_{v},
\Omega^{\gamma}_{u+v})$. For $i=1,\,2$ define the functions
$X_i(x,y)$ by the identity

\noindent\centerline{$\pi_i(e_x^{\delta_1},e_y^{\delta_2})=X_i(x,y)
e_{x+y}^{\gamma}$.}

By  Lemma \ref{Casimir}, there exists $(x_0,y_0)\in
u\times v$ such that

\noindent (1)\hskip6mm
$a_1(x,y)X_i(x+1,y-1)+b_1(x,y)X_i(x,y)+c_1(x,y)X_i(x-1,y+1)=0$,

\noindent (2)\hskip6mm
$a_2(x,y)X_i(x+2,y-2)+b_2(x,y)X_i(x,y)+c_2(x,y)X_i(x-2,y+2)=0$,

\noindent for $i=1,2$ and  all $(x,y)\in C(x_0,y_0)$. 
Moreover by Lemma \ref{tech}, we can assume that the vectors $(X_1(x+1,y),X_1(x,y+1))$ and
$(X_2(x+1,y),X_2(x,y+1))$ are linearly independant, for all $(x,y)\in C(x_0,y_0)$.

{\it First step:} We claim that 

\noindent  $a_2(x,y)b_1(x+1,y-1)c_1(x-1,y+1)c_1(x,y)=$
 
\hskip3cm $a_1(x,y)a_1(x+1,y+1)b_1(x-1,y+1)c_2(x,y)$,

\noindent for all $(x,y)\in\nc^2$, where the functions $a_i,\,b_i$ and $c_i$ are
defined in the previous section.

From now on, we assume that $(x,y)$ belongs to
$C(x_0+1,y_0+1)$. To simplify
the expressions in the proof, we  set 
$A_{il}=a_i(x+l,y-l)$, $B_{il}=b_i(x+l,y-l)$ and
$C_{il}=c_i(x+l,y-l)$, for any  $l\in \{-1,0,1\}$). 
Thus Identity (2) can be written as:

\noindent \centerline{
$A_{2,0}X_i(x+2,y-2)+*X_i(x,y)+C_{2,0}X_i(x-2,y+2)=0$,}

\noindent Here and in what follows, $*$ stands for a certain constant
whose explicit value is not important at this stage. Multiplying by
$A_{1,1}C_{1,-1}$, we obtain

\noindent (3){
$A_{2,0}C_{1,-1}[A_{1,1}X_i(x+2,y-2)]+*X_i(x,y)+$}

\hskip5cm $A_{1,1}C_{2,0}[C_{1,-1}X_i(x-2,y+2)]=0$,

\noindent
Note that $(x+1,y-1)$ and $(x-1,y+1)$ belong to $C(x_0,y_0)$. Using Relation (1) we
have

\noindent (4)\centerline{
$A_{1,1}X_i(x+2,y-2)+B_{1,1}X_i(x+1,y-1)+*X_i(x,y)=0$}

\noindent (5)\centerline{
$*X_i(x,y)+B_{1,-1}X_i(x-1,y+1)+C_{1,-1}X_i(x-2,y+2)=0$}

\noindent
With Relations (4) and (5) we can eliminate the terms 
$[A_{1,1}X_i(x+2,y-1)]$ and $[C_{1,-1}X_i(x-2,y+2)]$ in Relation (3). We obtain

\noindent (6) $A_{2,0}B_{1,1}C_{1,-1}X_i(x+1,y-1)+* X_i(x,y)+$ 

\hskip5cm $A_{1,1}B_{1,-1}C_{2,0}X_i(x-1,y+1)=0$.

Moreover Relation (1) can be written as

\noindent (7)\centerline{
$A_{1,0}X_i(x+1,y-1)+*X_i(x,y)+C_{1,0}X_i(x-1,y+1)=0$,}

Thus Relations (6) and (7) provide two linear equations connecting
$X_i(x+1,y-1)$, $X_i(x,y)$ and $X_i(x-1,y+1)$. Since 
$(x,y), (x+1,y-1)$ is an adjacent pair, it follows 
that the two triples $(X_1(x+1,y-1),X_1(x,y),X_1(x-1,y+1))$
and $(X_2(x+1,y-1),X_2(x,y),X_2(x-1,y+1))$ are linearly
independent. So the linear relations
(6) and (7) are proportional, which implies that

\noindent\centerline{$A_{2,0}B_{1,1}C_{1,-1}C_{1,0}=A_{1,0}A_{1,1}B_{1,-1}C_{2,0}$,}

\noindent or equivalently

(8) $a_2(x,y)b_1(x+1,y-1)c_1(x-1,y+1)c_1(x,y)=$
 
\hskip3cm $a_1(x,y)a_1(x+1,y+1)b_1(x-1,y+1)c_2(x,y).$

\noindent This identity holds for all $(x,y)\in C(x_0+1,y_0+1)$.
Since $C(x_0+1,y_0+1)$ is Zariski dense in $\nc^2$, Identity
(8) holds for any $(x,y)\in\nc^2$.

{\it Second step:} We claim that

\centerline{$\delta_1+\delta_2-\gamma-\delta_1^2-\delta_2^2+\gamma^2=0$.}

Assume otherwise and set
$\tau=\delta_1+\delta_2-\gamma-\delta_1^2-\delta_2^2+\gamma^2$.   
We have
$b_1(x,y)=\tau-2xy$ , therefore the polynomial $b_1$ is irreducible.
Observe that all irreducible factors of the left side of (8) are
degree 1 polynomials, except $b_1(x+1,y-1)$ and all 
irreducible factors of the right side  side of (8) are
degree 1 polynomials, except $b_1(x-1,y+1)$. Hence the irreducible factors
of both sides do not coincide, which proves the claim.

{\it Third step:} We claim that $\delta_1$ and $\delta_2$ belong
to $\{-1/2,0,1\}$. Using that $\delta_1+\delta_2-\gamma-\delta_1^2-\delta_2^2+\gamma^2=0$,
Identity (8) looks like

$(x+2\delta_1) (x+1)(x-1-\delta_1)(x-\delta_1) f(y)$

\hskip3cm $=(x+\delta_1)(x+1+\delta_1)(x-1)(x-2\delta_1) g(y)$

\noindent where $f(y)$ and $g(y)$ are some functions of $y$. Since 
$x+1$ is a factor of the left side of the identity, it follows that
$\delta_1=1,\,0$ or $-1/2$. The proof that 
$\delta_2=1,\,0$ or $-1/2$ is identical.

{\it Fourth step:} We claim that the case 
$\delta_1=\delta_2=-1/2$ is impossible. The Equations (6) and (7) can be
written as

\noindent (6) \centerline{$aX_i(x+1,y-1)+b X_i(x,y)+ *X_i(x-1,y+1)=0$,}

\noindent (7)\centerline{$cX_i(x+1,y-1)+dX_i(x,y)+*X_i(x-1,y+1)=0$,}

\noindent where $a,\,b,\,c$ and $d$ are explicit functions of $x$ and $y$
(as before, * denotes some functions which are irrelevant for the present computation). 
Using  that $\delta_1+\delta_2-\gamma-\delta_1^2-\delta_2^2+\gamma^2=0$ and  
a brute force computation, we obtain

\noindent \centerline{$ad-bc=9/32(1+2y)(2x-1)(2xy-1)$.} 

\noindent So the Equations (6) and (7) are not proportional, which contradicts that
$(X_1(x+1,y-1), X_1(x,y))$ and $(X_2(x+1,y-1), X_2(x,y))$ are independent.

{\it Final step:} If 
$\delta_1$ and $\delta_2$ belongs to $\{0,1\}$, then
$\gamma^2-\gamma=0$ i.e $\gamma=0$ or $1$ and the
triple $(\delta_1,\delta_2, \gamma)$ satisfies Assertion (i).
If $\delta_1=-1/2$, then $\delta_2=0$ or $1$ and 
$\gamma^2-\gamma=3/4$, i.e. $\gamma=-1/2$ or $3/2$ and the
triple $(\delta_1,\delta_2, \gamma)$ satisfies Assertion (ii).
Similarly if $\delta_2=-1/2$, the
triple $(\delta_1,\delta_2, \gamma)$ satisfies Assertion (iii).
\end{proof}

\subsection{The case $\dim \cG_{\bW}(M \times N,P)=2$}

Let $M,N,P\in{\cS}$ with $\Supp P= \Supp M+\Supp N$.

\begin{lemma}\label{germ_dim=2}
We have $\dim \cG_{\bW}(M \times N,P)=2$
iff $\deg M=\deg N=\deg P=\{0,1\}$.
\end{lemma}

\begin{proof}
Set $u=\Supp M$ and $v=\Supp N$. By Lemma \ref{Gsl(2)},
we have $\dim \cG_{\fsl(2)}(M \times N,P)\leq 2$
and therefore  $\dim \cG_{\bW}(M \times N,P)\leq 2$.

\noindent {\it First step:} Assume that 
$\deg M=\deg N=\deg P=\{0,1\}$. By Lemma \ref{lower} we have 
$\dim \cG_{\bW}(M \times N,P)\geq 2$. Thus
$\dim \cG_{\bW}(M \times N,P)=2$.

\noindent{\it Second step:} Set
$d^-=\dim\cG_{\bW}(\Omega^0_u\times\Omega^{-1/2}_v,\Omega^{-1/2}_{u+v})$
and $d^+=\dim\cG_{\bW}(\Omega^0_u\times\Omega^{-1/2}_v,\Omega^{3/2}_{u+v})$.
We claim that $d^+=d^-=1$.

By Lemma \ref{sl(2)-germs}, there is a $\fsl(2)$-equivariant isomorphism
$\phi:\cG(\Omega^{-1/2})\rightarrow \cG(\Omega^{3/2})$. 
By Lemma \ref{zero_intersection}, 
$\phi_*\,\cG_{\bW}(\Omega^0_u\times\Omega^{-1/2}_v,\Omega^{-1/2}_{u+v})$
and $\cG_{\bW}(\Omega^0_u\times\Omega^{-1/2}_v,\Omega^{3/2}_{u+v})$ are two subspaces
of $\cG_{\fsl(2)}(\Omega^0_u\times\Omega^{-1/2}_v,\Omega^{3/2}_{u+v})$ with trivial
intersection. Thus we have $d^++d^-\leq  2$.
However by Lemma \ref{lower}, we have $d^+\geq 1$ and $d^-\geq 1$.
It follows that $d^+=d^-=1$.

\noindent{\it Third step:} Conversely, assume that
$\dim \cG_{\bW}(M \times N,P)=2$. It follows that 
$\dim \tilde{\cG}_{\fsl(2)}(M \times N,P)=2$. 

By the previous step, the case (ii) of the  assertion of Lemma 
\ref{dim_tilde_cG} cannot occur. Using the $\fS_2$-symmetry, the case (iii) is 
excluded as well. It follows that $\deg M=\deg N=\deg P=\{0,1\}$.

\end{proof}

\section{Determination of $\cG_{\bW}(M\times N,P)$}
\label{sect_matrix-M}

Let $M,N$ and $P\in\cS$. In this section, we will compute the space
$\cG_{\bW}(M\times N,P)$.

We will always assume that $\Supp P=\Supp M+\Supp N$, otherwise it is obvious
that $\cG_{\bW}(M\times N,P)=0$. In the previous section,
it has been shown that $\dim \cG_{\bW}(M\times N,P)\leq 2$,
and the case $\dim \cG_{\bW}(M\times N,P)=2$ has been determined.
So it remains to decide when $\cG_{\bW}(M\times N,P)$ is zero or not.

The final result is very simple to state, because
 $\dim \cG_{\bW}(M\times N,P)$ only depends on
$\deg M,\deg N$ and $\deg P$. 

\subsection {Upper bound for $\dim \cG_{\bW}(M\times N,P)$}

\begin{lemma}
Let $M,N,P$ and $Q\in \cS$ and let $\phi\in\cG_{\fsl(2)}(P,Q)$. We have

\noindent\centerline{$\phi_*\tilde{\cG}_{\fsl(2)}(M\times N,P)
\subset \tilde{\cG}_{\fsl(2)}(M\times N,Q)$}
\end{lemma}

\begin{proof}
{\it Step 1:} Set $\Omega_1=L_0^2+L_0-L_{-1}L_1$ and
$\Omega_2=L_0^2+2L_0-L_{-2}L_2$. Indeed $\Omega_1$
is the Casimir element of $U(\fsl(2))$ and it acts as some scalar $c(X)$ on any 
$\bW$-module $X \in \cS$. It turns out
that $\Omega_2$ acts on $X$ as $4c(X)$.

{\it Step 2:} In order to prove the lemma, we can assume that
$\phi\neq 0$. Therefore $c(P)=c(Q)$ and $\Omega_2$ acts by the same scalar
on $P$ and on $Q$. Thus we get

\noindent\centerline{$\Omega_2\circ\phi=\phi\circ\Omega_2$.}

\noindent Since $\phi$ commutes with $L_0$, we get

\noindent\centerline{$(L_{-2}L_2)\circ\phi=\phi\circ(L_{-2}L_2)$.}

\noindent from which the lemma follows.
\end{proof}

\begin{lemma}\label{inequality}
 Let $\delta_1,\delta_2$ and $\gamma\in\nc$ and let
$u$ and $v$ be $\nz$-cosets. Assume that none of the following conditions is
satisfied

(i) $\gamma= 0,1/2$ or $1$, 

(ii) $\delta_1=-1/2$, $\delta_2\in\{0,1\}$ and $\gamma\in\{-1/2,3/2\}$,
 
(iii) $\delta_1\in\{0,1\}$,  $\delta_2=-1/2$ and $\gamma\in\{-1/2,3/2\}$.

Then we have:

\noindent\centerline
{$\dim {\cG}_{\bW}(\Omega^{\delta_1}_u\times\Omega^{\delta_2}_v,
\Omega^{\gamma}_{u+v})+
\dim {\cG}_{\bW}(\Omega^{\delta_1}_u\times\Omega^{\delta_2}_v,
\Omega^{1-\gamma}_{u+v})\leq 1$.}
\end{lemma}

\begin{proof} By Lemma \ref{sl(2)-germs}, there exists an isomorphism
$\phi:\cG_{\fsl(2)}(\Omega^{1-\gamma}_{u+v})\rightarrow
\cG_{\fsl(2)}(\Omega^{\gamma}_{u+v})$.
Obviously we have ${\cG}_{\bW}(\Omega^{\delta_1}_u\times\Omega^{\delta_2}_v,
\Omega^{\gamma}_{u+v})\subset 
\tilde{\cG}_{\fsl(2)}(\Omega^{\delta_1}_u\times\Omega^{\delta_2}_v,
\Omega^{\gamma}_{u+v})$ and by the
previous lemma we also have 
$\phi_*{\cG}_{\bW}(\Omega^{\delta_1}_u\times\Omega^{\delta_2}_v,
\Omega^{1-\gamma}_{u+v})\subset 
\tilde{\cG}_{\fsl(2)}(\Omega^{\delta_1}_u\times\Omega^{\delta_2}_v,
\Omega^{\gamma}_{u+v})$.

By condition (i) and Lemma \ref{zero_intersection}, the two subspaces
${\cG}_{\bW}(\Omega^{\delta_1}_u\times\Omega^{\delta_2}_v,
\Omega^{\gamma}_{u+v})$ and  $\phi_*{\cG}_{\bW}(\Omega^{\delta_1}_u\times\Omega^{\delta_2}_v,
\Omega^{1-\gamma}_{u+v})$ intersects trivially,
thus we have

\noindent $\dim {\cG}_{\bW}(\Omega^{\delta_1}_u\times\Omega^{\delta_2}_v,
\Omega^{\gamma}_{u+v})+
\dim {\cG}_{\bW}(\Omega^{\delta_1}_u\times\Omega^{\delta_2}_v,
\Omega^{1-\gamma}_{u+v})\leq$

\hskip7cm$\dim \tilde{\cG}_{\fsl(2)}(\Omega^{\delta_1}_u\times\Omega^{\delta_2}_v,
\Omega^{\gamma}_{u+v})$.

By conditions (i), (ii) and (iii), Lemma \ref{dim_tilde_cG},
 we have $\dim \tilde{\cG}_{\fsl(2)}(\Omega^{\delta_1}_u\times\Omega^{\delta_2}_v,
\Omega^{\gamma}_{u+v})\leq 1$, which proves the lemma.

\end{proof}

\subsection {Necessary condition for 
$\tilde{\cG}_{\fsl(2)}(\Omega^{\delta_1}_u\times\Omega^{\delta_2}_v,
\Omega^{\gamma}_{u+v})\neq 0$}\label{subsect_det}

Recall the notations of the previous section. For $k=1,\,2$, set

\begin{align*}
&a_k(x,y)=(x+k\delta_1)(y-k\delta_2), \\
&b_k(x,y)=k^2(\delta_1+\delta_2-\gamma-\delta_1^2-\delta_2^2+\gamma^2)-2xy, \\
&c_k(x,y)=(x-k\delta_1)(y+k\delta_2). 
\end{align*}

Given an auxiliary integer $l$, set
$A_{il}(x,y)=a_i(x+l,y-l)$, $B_{il}(x,y)=b_i(x+l,y-l)$ and
$C_{il}(x,y)=c_i(x+l,y-l)$ and  set

\[ \M=\begin{pmatrix}
   A_{1,5}(x,y) & B_{1,5}(x,y) & C_{1,5}(x,y) & 0 & 0 & 0 \\
   0 & A_{1,4}(x,y) & B_{1,4}(x,y) & C_{1,4}(x,y) & 0 & 0 \\
   0 & 0 & A_{1,3}(x,y) & B_{1,3}(x,y) & C_{1,3}(x,y) & 0 \\
   0 & 0 & 0 & A_{1,2}(x,y) & B_{1,2}(x,y) & C_{1,2}(x,y) \\
   A_{2,4}(x,y) & 0 & B_{2,4}(x,y) & 0 & C_{2,4}(x,y) & 0 \\
   0 & A_{2,3}(x,y) & 0 & B_{2,3}(x,y) & 0 & C_{2,3}(x,y) \end{pmatrix}, \]

\noindent Moreover set

\noindent\centerline{$\E_{\delta_1,\delta_2,\gamma}(x,y)=\det\M$.}

In what follows, we will consider $\E_{\delta_1,\delta_2,\gamma}(x,y)$ as a polynomial 
in the variables
$x$ and $y$, with parameters $\delta_1,\delta_2$ and $\gamma$. 

\begin{lemma}\label{vanishing_det}

 If $\tilde{\cG}_{\fsl(2)}(\Omega^{\delta_1}_u\times\Omega^{\delta_2}_v,
\Omega^{\gamma}_{u+v})\neq
0$ then $\E_{\delta_1,\delta_2,\gamma}(x,y)=0$, for all $(x,y)\in\nc^2$.
\end{lemma}

\begin{proof}
Assume that 
$\cG_{\bW}(\Omega^{\delta_1}_u\times\Omega^{\delta_2}_v,\Omega^{\gamma}_{u+v})\neq 0$.
Thus there exists a $L_0$-equivariant bilinear
map $\pi: \Omega^{\delta_1}_u\times\Omega^{\delta_2}_v\rightarrow\Omega^{\gamma}_{u+v}$
whose germ is a non-zero element of 
$\tilde{\cG}_{\fsl(2)}(\Omega^{\delta_1}_u\times\Omega^{\delta_2}_v,
\Omega^{\gamma}_{u+v})$. For $(x,y)\in u\times v$, define the scalar
$X(x,y)$ by the identity

\noindent\centerline{$\pi(e_x^{\delta_1},e_y^{\delta_2})=X(x,y)
e_{x+y}^{\gamma}$.} 

Set $\X(x,y)=(X(x+6,y-6), X(x+5,y-5),\dots,X(x+1,y-1))$.
Using Lemmas \ref{consecutive} and \ref{Casimir} there exists 
$(x_0,y_0)$ such that

(i) $(X(x+2,y-2),X(x+1,y-1))\neq 0$, and

(ii) $\M.^t\X(x,y)=0$,

\noindent for all $(x,y)\in C(x_0,y_0)$. The first assertion implies that
$\X(x,y)\neq 0$ for any $(x,y)\in C(x_0,y_0)$. Thus 
$\det\M$ vanishes on $C(x_0,y_0)$. Since $C(x_0,y_0)$ is
Zariski dense,  $\E_{\delta_1,\delta_2,\gamma}(x,y)=0$, for all $(x,y)\in\nc^2$.

\end{proof}

\subsection{Zeroes of the polynomials $p_{i,j}(\delta_1,\delta_2,\gamma)$}
\label{sect_poly-pij}

Define the polynomials $p_{i,j}(\delta_1,\delta_2,\gamma)$ by the identity

\noindent\centerline{$\E_{\delta_1,\delta_2,\gamma}(x,y)=
\sum_{i,j}p_{i,j}(\delta_1,\delta_2,\gamma) x^iy^j$.}

\noindent Since the entries of the matrix $\M$ are quadractic polynomials in
$x,y,\delta_1,\delta_2$ and $\gamma$, $\E_{\delta_1,\delta_2,\gamma}(x,y)$ is
a polynomial of degree $\leq 12$. Set

$C(\delta_1,\delta_2,\gamma)=(\delta_1+\delta_2+\gamma)(\delta_1+\delta_2-\gamma)
   (\delta_1+\delta_2+1-\gamma)(\delta_1+\delta_2-1+\gamma)$.

\begin{lemma}\label{divisibility}

(i) We have $p_{i,j}(\delta_1,\delta_2,\gamma)=p_{i,j}(\delta_1,\delta_2,1-\gamma)$,

(ii) Each polynomial $p_{i,j}(\delta_1,\delta_2,\gamma)$ is divisible by
$C(\delta_1,\delta_2,\gamma)$.
\end{lemma}

\begin{proof} 
All entries of the matrix $\M$ are invariant under the involution 
$\gamma\mapsto 1-\gamma$, so we have

\noindent\centerline{$\E_{\delta_1,\delta_2,\gamma}(x,y)=\E_{\delta_1,\delta_2,1-\gamma}(x,y)$}

\noindent which implies the first assertion.

It follows from Lemma \ref{lower} that
$\cG_{\bW}(\Omega^{\delta_1}_u\times\Omega^{\delta_2}_v,\Omega^{\gamma}_{u+v})\neq 0$
whenever $\gamma=\delta_1+\delta_2$ or $\gamma=\delta_1+\delta_2+1$.

Hence by Lemma \ref{vanishing_det}, as a polynomial in $\delta_1,\delta_2,\gamma,x$ and $y$,
$\E_{\delta_1,\delta_2,\gamma}(x,y)$ is divisible by
$D(\delta_1,\delta_2,\gamma)$, where $D(\delta_1,\delta_2,\gamma)
=(\delta_1+\delta_2-\gamma) (\delta_1+\delta_2+1-\gamma)$. By the first assertion,
it is also divisible by 
$D(\delta_1,\delta_2,1-\gamma)$. Since we have 

\noindent\centerline{$C(\delta_1,\delta_2,\gamma)=
D(\delta_1,\delta_2,\gamma)D(\delta_1,\delta_2,1-\gamma)$}

\noindent  each $p_{i,j}$ is divisible $C(\delta_1,\delta_2,\gamma)$.
\end{proof}

Denote by  $\tau$  the involution 
$(\delta_1,\delta_2,\gamma)\mapsto(\delta_1,\delta_2,1-\gamma)$.
Let $\fZ\subset \nc^3$ be the following set

$\fZ=(\bigcup_{0\leq i\leq 1} H_i\cup H_i^\tau)
\bigcup (\bigcup_{1\leq i\leq 4} D_i\cup D_i^\tau)
\bigcup (\bigcup_{1\leq i\leq 2} \{ P_i\cup P_i^\tau)\}$,

\noindent where the planes $H_i$, the lines $D_i$ and the points
$P_i$ are defined in Section \ref{sect_example-germ}. For a polynomial 
$f$, denote by $Z(f)$ its zero set.

\begin{lemma}\label{zero_pij}
We have $Z(p_{1,3})\cap Z(p_{3,1})\cap Z(p_{2,2})\subset\fZ$.
\end{lemma}

The proof requires the explicit computation of $\E_{\delta_1,\delta_2,\gamma}(x,y)=\det\M$, 
and we have used MAPLE for  this purpose.  The next proof contains the explicit expressions of
$p_{1,3}, p_{3,1}$ and $p_{2,2}$. The whole expression for $\det\M$ is given in Appendix A.

\begin{proof} 
{\it Step 1:}
By Lemma \ref{divisibility}, there are polynomials $q_{ij}$ such that

\noindent\centerline{$p_{ij}(\delta_1,\delta_2,\gamma)=C(\delta_1,\delta_2,\gamma)
q_{ij}(\delta_1,\delta_2,\gamma)$.}

\noindent Since $Z(C)$ is the union of the four planes $H_0, H_0^\tau, H_1, H_1^\tau$,
it remains to prove that   
$Z(q_{1,3})\cap Z(q_{3,1})\cap Z(q_{2,2})\subset\fZ$.

Using MAPLE, it turns out that

\vspace{-0.1in}
\begin{align*}
q_{2,2}=
&\gamma^2(1-\gamma)^2+2\gamma(1-\gamma)(\delta_1^2+\delta_2^2-2\delta_1
\delta_2-2(\delta_1+\delta_2)+4)\\
&+(\delta_1^4+\delta_2^4-4\delta_1\delta_2(\delta_1^2+
\delta_2^2)+38\delta_1^2\delta_2^2)-4(\delta_1+\delta_2)^3\\
&-(13(\delta_1^2+\delta_2^2)-6\delta_1\delta_2)+4(\delta_1+\delta_2)+12,
\end{align*}

\vspace{-0.3in}
\begin{align*} 
-\frac{1}{8}q_{1,3}=
&\delta_1(\delta_1-1)[\gamma(1-\gamma)
+\delta_1^1+\delta_2^2-4\delta_1\delta_2+3(\delta_1-\delta_2)+2].\\
\end{align*}

\vspace{-0.3in}
\begin{align*}
-\frac{1}{8}q_{3,1}=
&\delta_2(\delta_2-1)[\gamma(1-\gamma)
+\delta_1^1+\delta_2^2-4\delta_1\delta_2-3(\delta_1-\delta_2)+2], \\
\end{align*}

{\it Step 2:} The previous expressions provide (miraculous) factorizations

\noindent\centerline{$-\frac{1}{8} q_{1,3}=L_1 L_2 Q$ and $-\frac{1}{8} q_{3,1}=L'_1 L'_2 Q'$}

\noindent  where $L_1, L_2,L'_1,$ and $L'_2$ are degree one polynomials and $Q$ and $Q'$ are
quadratic polynomials. We have to 
prove that  $Z(P)\cap Z(P')\cap Z(q_{2,2})\subset \fZ$ for any  factor
$P$ of $q_{1,3}$ and any factor $P'$ of $q_{3,1}$. 
This amonts to 9 cases, which will be treated seperately.

{\it Step 3: proof that $Z(L_i)\cap Z(L'_j)\cap Z(q_{2,2})\subset \fZ$,
$\forall i,j\in\{1,2\}$.}

We claim that, in each case, the intersection consists of $4$ points lying in
$\fZ$. Since the four cases
are similar, we will  only consider the case where the first factor is $\delta_1$ and 
the second one is $\delta_2$.  

For a point $(0,0,\gamma)\in Z(\delta_1)\cap Z(\delta_2)\cap Z(q_{2,2})$, we have

\noindent\centerline{
$q_{2,2}(0,0,\gamma)=\gamma^2(1-\gamma)^2 + 8 \gamma(1-\gamma)+12$=0.}

\noindent Thus we have $\gamma(1-\gamma)=-2$ or $-6$. It follows that
$Z(\delta_1)\cap Z(\delta_2)\cap Z(p_{2,2})$
consists of the four points $(0,0,-2), (0,0,-1),(0,0,2),(0,0,3)$ which are 
all in $\fZ$.

{\it Step 4: proof that $Z(L_i)\cap Z(Q')\cap Z(q_{2,2})\subset \fZ$,
$\forall i\in\{1,2\}$.}

More precisely, we claim that the planar quadric $Z(L_i)\cap Z(Q')$ consists of two lines which
are both in $\fZ$. Since the two cases are similar, we will just treat
the case where the factor $L_i$ is $\delta_1$. 
We have 

$Q'(0,\delta_2,\gamma)=\gamma(1-\gamma)+\delta_2^2+3\delta_2+2$

\noindent \hskip2.65cm$=-(\gamma+\delta_2+1)(\gamma-\delta_2-2),$

\noindent which proves the claim. It follows that
$Z(L_i)\cap Z(Q')\cap Z(q_{2,2})\subset \fZ$.

{\it Step 5: Proof that $Z(Q)\cap Z(L_j')\cap Z(q_{2,2})\subset \fZ$,
$\forall j\in\{1,2\}$.}

This case is identical to the previous one.

{\it Step 6: Proof that $Z(Q)\cap Z(Q')\cap Z(q_{2,2})\subset \fZ$.} 

Indeed $Q'-Q$ is a scalar multiple of $\delta_1-\delta_2$ and therefore
$Z(Q)\cap Z(Q')$ is (again miraculously) a planar quadric. 

We have $Q(\delta,\delta,\gamma)= \gamma(1-\gamma)-2\delta^2+2$,
so $Z(Q)\cap Z(Q')$ is the sets of all $(\delta,\delta,\gamma)\in\nc^3$
such that $\gamma(1-\gamma)=2\delta^2-2$. Since $q_{2,2}(\delta,\delta,\gamma)$ is a 
polynomial in $\delta$ and $\gamma(1-\gamma)$ we can 
eliminate $\gamma(1-\gamma)$. We have

\noindent\centerline{
$q_{2,2}(\delta,\delta,\gamma)=12\delta(3\delta+2)(\delta-1)^2$}

\noindent for any $(\delta,\delta,\gamma)\in Z(Q)\cap Z(Q')$. It follows that
$Z(Q)\cap Z(Q')\cap Z(q_{2,2})$ consists of the $6$ points
$(1,1,0), (1,1,1), (0,0,-1),(0,0,2),(-2/3,-2/3,-2/3)$ and $(-2/3,-2/3,5/3)$.
Since there are all in $\fZ$, the proof is complete.

\end{proof}

With more care, it is easy to prove that $\bigcap Z(p_{i,j})$ is precisely $\fZ$
but this is not necessary for what follows.

\subsection{Determination of $\cG_{\bW}(M\times N,P)$}

Recall that $\fz^*$ the set of all $(\delta_1,\delta_2,\gamma)\in\fz$ such that
$\{\delta_1,\delta_2,\gamma\}\not\subset\{0,1\}$. 
Let $M$, $N$ and $P$ be in $\cS$. In order to determine $\cG_{\bW}(M\times N,P)$
we will always assume that

\noindent\centerline {$\Supp P=\Supp M+\Supp N$.}

\noindent Otherwise $\cG_{\bW}(M\times N,P)$ would be obviously zero.
Next let $\delta_1\in\deg M$, $\delta_2\in\deg N$ and $\gamma\in\deg P$.

\begin{thm}\label{thm2} 
We have

(i) $\dim \cG_{\bW}(M\times N,P)=2$ if $\{\delta_1,\delta_2,\gamma\}\subset\{0,1\}$,
and the maps $\pi_1$, $\pi_2$ of Lemma \ref{lower} form a basis of this space,

(ii) $\dim \cG_{\bW}(M\times N,P)=1$ if $(\delta_1,\delta_2,\gamma)\in \fz^*$
and the map $\pi$ of Table \ref{table1} provides a  generator of this space,

(iii) $\dim \cG_{\bW}(M\times N,P)=0$ otherwise.
\end{thm}

\begin{proof} Set $u=\Supp M$ and $v=\Supp N$. By Lemmas \ref{W-germs} and \ref{lemma_germ-gen}, we can assume that
$M=\Omega_u^{\delta_1}$, $N=\Omega_v^{\delta_2}$ and 
$P=\Omega_{u+v}^{\gamma}$.

{\it Step 1:} We claim that  $(\delta_1,\delta_2,\gamma)$ belongs to $\fz$ 
if $\cG_{\bW}(M\times N,P)\neq 0$. 

Assume that $\cG_{\bW}(M\times N,P)\neq 0$. Since
$\tilde{\cG}_{\fsl(2)}(M\times N,P)\neq 0$ it follows from Lemmas
\ref{vanishing_det} and \ref{zero_pij} that 
$(\delta_1,\delta_2,\gamma)$ belongs to $\fZ$. It
is clear from its definition that  $\fZ\subset \fz\cup\fz^\tau$. Hence 
$(\delta_1,\delta_2,\gamma)$ or $(\delta_1,\delta_2,1-\gamma)$ belongs
to $\fz$. 

When $(\delta_1,\delta_2,1-\gamma)\notin\fz$  the claim is proved. 
 Moreover if $\gamma=0,1/2$ or $1$, we have
$\cG_{\bW}(\Omega_{u+v}^{\gamma})=\cG_{\bW}(\Omega_{u+v}^{1-\gamma})$ and thus
$\cG_{\bW}(\Omega_u^{\delta_1}\times \Omega_v^{\delta_2}, \Omega_{u+v}^{1-\gamma})
=\cG_{\bW}(\Omega_u^{\delta_1}\times \Omega_v^{\delta_2}, \Omega_{u+v}^{\gamma})$, which proves
the claim in this case. Therefore, we can assume that 
$(\delta_1,\delta_2,1-\gamma)\in \fz$ and that $\gamma\notin\{0,1/2,1\}$.

By Lemma \ref{lower}, we have
$\cG_{\bW}(\Omega_u^{\delta_1}\times \Omega_v^{\delta_2}, \Omega_{u+v}^{1-\gamma})\neq 0$.
Therefore  it follows that

\noindent\centerline{
$\dim \cG_{\bW}(\Omega_u^{\delta_1}\times \Omega_v^{\delta_2}, \Omega_{u+v}^{\gamma})
+\dim \cG_{\bW}(\Omega_u^{\delta_1}\times \Omega_v^{\delta_2}, \Omega_{u+v}^{1-\gamma})
\geq 2$.}

\noindent  By Lemma \ref{inequality} we have

(i) $\delta_1=-1/2$, $\delta_2 \in \{0,1\}$ and $\gamma\in\{-1/2,3/2\}$, or 

(ii) $\delta_1 \in\{0,1\}$, $\delta_2=-1/2$, and $\gamma\in\{-1/2,3/2\}$.

\noindent These $8$ possible triples for $(\delta_1,\delta_2,\gamma)$ belong
to $\fz$ and therefore the claim is proved.

{\it Step 2:} Assertion (i) follows from Lemma \ref{germ_dim=2}. From now on,  we can assume that
$\{\delta_1,\delta_2,\gamma\}\not\subset\{0,1\}$.
It follows that 
$\dim \cG_{\bW}(\Omega_u^{\delta_1}\times \Omega_v^{\delta_2}, 
\Omega_{u+v}^{\gamma})=0$ or $1$. In particular Assertion (ii) and (iii) are equivalent
and it is enough to prove the first one.

If $(\delta_1,\delta_2,\gamma)\in \fz^*$ we have
$\cG_\bW(\Omega_u^{\delta_1}\times \Omega_v^{\delta_2}, 
\Omega_{u+v}^{\gamma})\neq 0$ by Lemma \ref{lower} and therefore
$\dim \cG_\bW(\Omega_u^{\delta_1}\times \Omega_v^{\delta_2}, 
\Omega_{u+v}^{\gamma})=1$. Conversely if
$\dim \cG_\bW(\Omega_u^{\delta_1}\times \Omega_v^{\delta_2}, 
\Omega_{u+v}^{\gamma})=1$ it follows from the previous step that
$(\delta_1,\delta_2,\gamma)$ belongs to  $\fz^*$. Thus assertion
(ii) is proved.
\end{proof}

\section{On the map $\B_{\bW}(M\times N,P)\rightarrow\cG_{\bW}(M\times N,P)$}

Let $M,N$ and $P$ be in $\cS$. 
The space $\cG_{\bW}(M\times N,P)$ has been determined by Theorem 2.
In particular $\cG_{\bW}(M\times N,P)$ has always dimension $0$, $1$ or $2$.
In the  Sections 8-10, we determine  which germs $\mu\in \cG_{\bW}(M\times N,P)$
can be lifted to a $\bW$-equivariant bilinear map
$\pi:M\times N\rightarrow P$. Since the final result contains many
particular case, it has been split into two parts. Indeed  Theorem 3.1
(in Section 9) involves  the case where $\cG_{\bW}(M\times N,P)$ has dimension one, and 
and Theorem 3.2 (in Section 10)  involves  the case where $\cG_{\bW}(M\times N,P)$ has dimension
two. 
In this section, we recall general facts and conventions used in Sections 9 and 10.

\subsection{Germs and $\fS_3$-symmetry}
Let $M,N,P\in{\cS}$. Recall the exact sequence:

\noindent\centerline{
$0\rightarrow \B_{\bW}^0(M\times N,P)\rightarrow
\B_{\bW}(M\times N,P)\rightarrow \cG_{\bW}(M\times N,P)$.}
 
Determining the image of the map
$\B_{\bW}(M\times N,P)\rightarrow \cG_{\bW}(M\times N,P)$
is easy, but it requires a very long case-by-case analysis.
It would be pleasant 
to use the $\fS_3$-symmetry to reduce the number of cases. Unfortunately
the definition of $\cG_{\bW}(M\times N,P)$ is not $\fS_3$-symmetric. 
However, set

\noindent\centerline{
$\overline{\B}_{\bW}(M\times N,P)= \B_{\bW}(M\times N,P)/\B^{0}_{\bW}(M\times N,P)$.}

\begin{lemma} For any $M,N$ and $P\in\cS$, we have

\noindent\centerline{
$\overline{\B}_{\bW}(M\times N,P^*)=\overline{\B}_{\bW}(M\times P,N^*)$}
\end{lemma}

\begin{proof}
By Lemma \ref{germ-deg}, the space $\B^{0}_{\bW}(M\times N,P^*)$ is exactly
the space of degenerate $\bW$-equivariant maps from $M\times N$ to $P^\ast$. 
Hence $\B^{0}_{\bW}(M\times N,P^*)$ and $\overline{\B}_{\bW}(M\times N,P^*)$ are
fully symmetric in $M,N$ and $P$.
\end{proof}

\subsection{List of cases for the proof of Theorem 3}
Start with a general result:
\begin{lemma}\label{lemma_simple}
Let  $M,N$ and $P\in\cS$ be irreducible $\bW$-modules. We have

\noindent\centerline{$\B_\bW(M\times N,P)\simeq\cG_{\bW}(M\times N,P)$.}
\end{lemma}
 
\begin{proof} Looking at Table \ref{table2}, it is clear that  
$\B^0_\bW(M\times N,P)=0$ whenever $M,N$ and $P$ are irreducible.
Moreover, it is clear from Table \ref{table1} that any germ
can be lifted.
\end{proof}

Let $M,N$ and $P\in \cS$. In order to determine $\overline{\B}_{\bW}(M\times N,P)$,
we will always tacitely assume that $\cG_{\bW}(M\times N,P)\neq 0$. By the previous lemma,
we can assume that at least one module is reducible, i.e. in the $AB$-family.
As usual, we will assume that all modules are indecomposable. Using the 
$\fS_3$-symmetry, we can reduce to the following 6 cases:

\begin{enumerate}
\item $\deg M=\deg N=\{0,1\}$ and $\deg P=2$,
\item $\deg M=\deg N=\{0,1\}$ and $\deg P=3$,
\item $\deg M=\{0,1\}$, $\deg N=\delta$ and $\deg P=\delta$ with 
$\delta \in \nc \setminus \{0,1\}$.
\item $\deg M=\{0,1\}$, $\deg N=\delta$ and $\deg P=\delta+1$ with
 $\delta \in \nc \setminus \{0,1\}$.
\item $\deg M=\{0,1\}$, $\deg N=\delta$ and $\deg P=\delta+2$ with 
$\delta \in \nc \setminus \{0,1\}$
\item $\deg M=\deg N=\deg P=\{0,1\}$.
\end{enumerate}

The cases case 1-5 are treated in Section 9. In this case, we have $\dim\cG_\bW(M\times N,P)=1$,
so it is enough to decide if ${\overline\B}_\bW(M\times N,P)$ is zero or not. The
case 6 is treated in section 10. In this case,
$\cG_\bW(M\times N,P)$ is two dimensional and the  analysis is more involved.

\subsection{Typical argument for the proof of Theorem 3}
Let $M \in \cS$ and let $u$ be its support. In Sections 9 and 10, we will denote by $(e_x^M)_{x \in u}$ a basis of $M$ as in Section \ref{KS_Theorem}.

The proofs of Theorems 3.1 and 3.2 are given by several lemmas and a repeated procedure,
that we call an \textit{argument by restriction}, which is described as follows. 

For an integer $d \in \nz_{>1}$, the subalgebra $\bW^{(d)}:=\bigoplus_{n} \nc L_{dn}$ is isomorphic to $\bW$. 
Let $M$ be a $\bW$-module in the class $\cS$ and let $x \in \Supp\,M$. 
The subspace
\[ M_d(x):=\bigoplus_{\substack{y \in u \\ x-y \in d\nz}} M_y \]
of $M$ is a $\bW^{(d)}$-module. Moreover, when $x \not\in d\nz$, $M_d(x)$ is irreducible. 

Now, let $M,N$ and $P$ be $\bW$-modules in the class $\cS$, let $x \in \Supp\,M,~y \in \Supp\, N$ and let $\overline{\pi} \in\cG_{\bW}(M\times N,P)$. Since $\cG_{\bW^{(d)}}(X\times Y,Z)\simeq \B_{\bW^{(d)}}(X\times Y,Z)$ whenever $X,Y$ and $Z$ are irreducible $\bW^{(d)}$-modules of the class $\cS$, $\overline{\pi}$ has unique lifting $\pi$ to $M_d(x)\times N_d(y)$ whenever $x,y,x+y \not\in d\nz$. 
Hence, varying $d$ and $x,y$, we see that $\pi(e_x^M,e_y^N)$ is uniquely determined by $\overline{\pi}$ whenever $x,y,x+y\neq 0$.

\section{Computation of $\overline{\B}_{\bW}(M\times N,P)$ when
$\dim\cG_{\bW}(M\times N,P)=1$}

\subsection{The Theorem 3.1}

The following table provides a list of triples $(M,\,N,\,P)$  of $\bW$-modules of
the class  $\cS$ and one non-zero element
$\pi\in {\overline\B}_{\bW}(M\times N,P)$. 
Since for each entry $(M,\,N,\,P)$
of Table \ref{table4} we have  $\deg M$ or $\deg N$ or $\deg P\neq \{0,1\}$, the
corresponding   $\pi$ is a basis of
${\overline\B}_{\bW}(M\times N,P)$.

In what follows, we  denote by $d_\xi$ and $d^\xi$ the natural maps:

\centerline{$d^\xi:\Omega^0_0\rightarrow A_\xi$ and
$d_\xi: B_\xi\rightarrow \Omega^1_0$.}

\begin{table}[h]\caption{\textbf{Non-zero ${\overline\B}_{\bW}(M\times N,P)$,
when $\deg M, \,\deg N$ or $\deg P\neq \{0,1\}$}}
\label{table4}

\begin{center}
\begin{tabular}{|c|c|c||c|} \hline

   & $M\times N$ or $N\times M$ & $P$ & $\pi$   \\ \hline

1.& $\Omega^{\delta_1}_u\times\Omega^{\delta_2}_{v}$ & 
$\Omega^{\delta_1+\delta_2}_{u+v}$& $P^{\delta_1,\delta_2}_{u,v}$\\ \hline

2.& $\Omega^{\delta_1}_u\times\Omega^{\delta_2}_{v}$ & 
$\Omega^{\delta_1+\delta_2+1}_{u+v}$&$B^{\delta_1,\delta_2}_{u,v}$
\\ \hline

3.& $A_\xi\times\Omega^{\delta}_{u}$ & 
$\Omega^{\delta+1}_{u}$&$B^{0,\delta}_{0,u}(\xi)$
\\ \hline

4.& $\Omega^{\delta}_{u}\times \Omega^{-\delta}_{-u}$ & 
$B_\xi$&$B^{\delta,-\delta}_{u,-u}(\xi)$
\\ \hline

5.& $\Omega^{-2/3}_u\times \Omega^{-2/3}_{v}$ & 
$\Omega^{5/3}_{u+v}$&$G_{u,v}$  \\
\hline

6.& $B_\xi\times\Omega^{\delta}_{u}$ & 
$\Omega^{\delta+1}_{u}$&$P^{1,\delta}_{0,u}\circ (d_\xi\times id)$
\\ \hline 

7.& $\Omega^{\delta}_{u}\times \Omega^{-\delta}_{-u}$ & 
$A_\xi$&$\d^\xi\circ P^{\delta,-\delta}_{u,-u}$
\\ \hline 

8.& $B_\xi\times \Omega^{\delta}_{u}$ & 
$\Omega^{\delta+2}_{u}$&$B^{1,\delta}_{0,u}\circ (d_\xi\times id)$ 
 \\ \hline

9.& $\Omega^{\delta}_{u}\times \Omega^{-\delta-1}_{-u}$ & 
$A_\xi$&$\d^\xi\circ B^{\delta,-\delta-1}_{u,-u}$ 
 \\ \hline

10.& $A_{\eta}\times B_{\xi}$ & 
$\Omega^{2}_{0}$&$B^{0,1}_{0,0}(\eta)\circ(id\times d_{\xi})$  \\ \hline

11.& $A_{\eta}\times \Omega^{-1}_0$ &   $A_{\xi}$&
$\d^{\xi}\circ B^{0,-1}_{0,0}(\eta)$ 
\\ \hline

12. & $B_{\xi}\times \Omega^{-1}_0$ &   $B_{\eta}$&
$B^{1,-1}_{0,0}(\eta)(d_{\xi}\times id)$ \\ \hline

 13.& $B_{\eta}\times B_\xi$ &   $\Omega^{2}_{0}$&
$P^{1,1}_{0,0}\circ(d_{\eta}\times d_{\xi})$    \\
\hline

14.& $B_{\eta}\times \Omega^{-1}_0$&   $A_{\xi}$&
$d^{\xi}\circ P^{1,-1}_{0,0}\circ(d_{\eta}\times id)$   \\
\hline

15.&$B_{\eta}\times B_{\xi}$ & $\Omega^3_0$ & $B^{1,1}_{0,0}\circ(d_{\eta}\times
d_{\xi})$\\ \hline

16.&$B_{\eta}\times \Omega^{-2}_0$ & $A_\xi$ & 
$d^\xi\circ B^{1,-2}_{0,0}\circ(d_{\eta}\times id)$\\ \hline

\end{tabular}
\end{center}

\vbox{\small{ 
The degree condition implies the following restrictions:
$\{\delta_1,\delta_2,\delta_1+\delta_2\}\not\subset\{0,1\}$ in line 1,
$(\delta_1,\delta_2)\neq (0,0)$ in line 2, 
$\delta\neq 0$ in lines 3 and 4, and 
$\delta\neq 0,\,1$ in lines 6 and 7.}}

\end{table}

Conversely, we have

\begin{thm}\hskip-1.7mm{\bf 1} 
Let $M, N$ and $P$ be $\bW$-modules of the class 
$\cS$ with $\deg M$ or $\deg N$ or $\deg P\neq \{0,1\}$.
Then ${\overline\B}_{\bW}(M\times N,P)$ has dimension one if the triple
$(M,N,P)$ appears in Table \ref{table4}. Otherwise, we have
${\overline\B}_{\bW}(M\times N,P)=0$.
\end{thm}

\subsection{The case $\deg M=\deg N=\{0,1\}$ and $\deg P=2$}
In this case, there can be five subcases as follows:

\begin{enumerate}
\item $(M,N,P)=(A_{\xi},\Omega_u^0,\Omega_u^2)$ 
         with $u \not\equiv 0 \mod \nz$,

\item $(M,N,P)=(B_{\xi},\Omega_u^0,\Omega_u^2)$ 
         with  $u \not\equiv 0 \mod \nz$,

\item $(M,N,P)=(A_{\eta},A_{\xi},\Omega_{0}^2)$,

\item $(M,N,P)=(A_{\eta},B_{\xi},\Omega_{0}^2)$,

\item $(M,N,P)=(B_{\eta},B_{\xi},\Omega_{0}^2)$.
\end{enumerate}

\begin{lemma}\label{prop_002}
Let $M,N,P$ be $\bW$-modules of the class $\cS$ as above. Then
$\overline{\B}_{\bW}(M\times N,P)$ is trivial iff
$(M,N,P)=(A_{\eta},A_{\xi},\Omega_{0}^2)$ with 
$\xi\neq \infty$ and $\eta\neq\infty$
\end{lemma}

\begin{proof}
By Table \ref{table4}, except for the case 3 with $\eta=\infty$ or $\xi=\infty$, we have $\dim \overline{\B}_{\bW}(M\times N,P)=1$. 
Hence, we show the proposition in the case $(M,N,P)=(A_\eta,A_\xi,\Omega_0^2)$. 

Let $(a,b)$ and $(c,d)$ be projective coordinates of $\eta, \xi \in \np^1$, respectively, and let $\{e_m^M\}, \{e_m^N\}$ and $\{e_m^P\}$ be basis of $A_{a,b}, A_{c,d}$ and $\Omega_0^2$, respectively, as in Section 1.1. Assume that $\overline{\B}_{\bW}(M\times N,P)\neq 0$. By an argument by restriction, one sees that there exists $\pi \in \B_{\bW}(M \times N,P)$ such that $\pi(e_m^M,e_n^N)=e_{m+n}^P$ whenever $m,n,m+n\neq 0$.
It can be checked that this formula extends to the case $m,n \neq 0$.
Set $\pi(e_m^M,e_0^N)=X_1(m)e_m^P$ and $\pi(e_0^M,e_m^N)=X_2(m)e_m^P$ for $m\neq 0$. Calculating $L_n.\pi(e_m^M,e_0^N)$ for $n=1,2$, one obtains \\
\centerline{$(m+2n)X_1(m)=(m+n)X_1(m+n)+cn^2+dn$} 

\noindent from which we have
$X_1(m)=-cm+d$. Similarly, one also obtains $X_2(m)=-am+b$.
By calculating $L_m.\pi(e_0^M,e_0^N)$, we obtain $ac=0$. 
\end{proof}

\subsection{The case $\deg M=\deg N=\{0,1\}$ and $\deg P=3$}
In this case, there can be five subcases as follows:

\begin{enumerate}
\item $(M,N,P)=(A_\xi,\Omega_u^0,\Omega_u^3)$  
         with $u \not\equiv 0 \mod \nz$,

\item $(M,N,P)=(B_\xi,\Omega_u^0,\Omega_u^3)$ 
         with $u \not\equiv 0 \mod \nz$,

\item $(M,N,P)=(A_{\xi},A_{\eta},\Omega_0^3)$,

\item $(M,N,P)=(A_{\xi},B_{\eta},\Omega_0^3)$,

\item $(M,N,P)=(B_{\xi},B_{\eta},\Omega_0^3)$.
\end{enumerate}

\begin{lemma}\label{prop_003}
Let $M,N,P$ be $\bW$-modules as above. Then
${\overline \B}_{\bW}(M\times N,P)$ is trivial 
iff $(M,N,P)$ is one of the following types:

\begin{enumerate}
\item $(M,N,P)=(A_\xi,\Omega_u^1,\Omega_u^3)$ 
with $\xi\neq\infty$ and $u \not\equiv 0 \mod \nz$,

\item $(M,N,P)=(A_\xi,A_\eta,\Omega_0^3)$ 
with $(\xi,\eta) \neq (\infty,\infty)$,

\item $(M,N,P)=(A_{\xi},B_\eta,\Omega_0^3)$ 
with $\xi\neq\infty$.
\end{enumerate}
\end{lemma}

\begin{proof}
By Table \ref{table4}, it is clear that ${\overline \B}_{\bW}(M\times N,P)$
is trivial only if $(M,N,P)$ is one of the three cases in Lemma \ref{prop_003}.
Hence, it is sufficient to show that, for these three cases, ${\overline \B}_{\bW}(M\times N,P)$ is trivial.

The proof for the first and third cases are similar to that of Lemma \ref{prop_002}
and is left to the reader. The second case can be proved as follows.

Choose any $\tau\in\np^1$. Any non-degenerate $\pi \in \B_{\bW}(A_{\eta}\times
A_{\xi},\Omega_0^3)$ induces a non-degenerate bilinear map $\pi' \in \B_{\bW}(A_\eta \times B_\tau,\Omega_0^3)$, by composing with the map $B_\tau \rightarrow \overline{A} \hookrightarrow A_\xi$. It follows from the first case that $\eta=\infty$.
Similarly, we have $\xi=\infty$.
\end{proof}

\subsection{The case $\deg M=\{0,1\}$, $\deg N=\delta$ and $\deg P=\delta+2$ with 
$\delta \in\nc \setminus \{0,1\}$} 

In this case, there can be two subcases as follows:
\begin{enumerate}
\item $(M,N,P)=(A_\xi,\Omega_u^\delta,\Omega_u^{\delta+2})$, 
         
\item $(M,N,P)=(B_\xi,\Omega_u^\delta,\Omega_u^{\delta+2})$.
\end{enumerate}

\begin{lemma}\label{prop_0dd+2}
Let $M,N,P$ be $\bW$-modules as above. Then, 
${\overline\B}_{\bW}(M\times N,P)$ is trivial iff
$M=A_\xi$ with $\xi\neq\infty$ and $\delta\neq -1$. 
\end{lemma}

The proof  is similar to that of Lemma \ref{prop_002}.

\subsection{The case $\deg M=\{0,1\}$, $\deg N=\delta$ and $\deg P=\delta+1$ with $\delta \in
\nc \setminus \{0,1\}$} 

In this case, there can be two subcases as follows:

\begin{enumerate}
\item $(M,N,P)=(A_\xi,\Omega_u^\delta,\Omega_u^{\delta+1})$,

\item $(M,N,P)=(B_\xi,\Omega_u^\delta,\Omega_u^{\delta+1})$.
\end{enumerate}

Looking at the lines 3 and 6 in Table \ref{table4}, we obtain

\begin{lemma}\label{prop_0dd+1}
Let $M,N,P$ be $\bW$-modules of the class $\cS$ as above. Then
$\dim {\overline\B}_{\bW}(M\times N,P)=1$.
\end{lemma}

\subsection{The case $\deg M=\{0,1\}$, $\deg N=\delta$ and $\deg P=\delta$ 
with $\delta \in \nc \setminus \{0,1\}$}
In this case, there can be two subcases as follows:

\begin{enumerate}
\item $(M,N,P)=(A_{\xi},\Omega_u^\delta,\Omega_u^\delta)$,

\item $(M,N,P)=(B_\xi,\Omega_u^\delta,\Omega_u^\delta)$.
\end{enumerate}

\begin{lemma}\label{prop_0dd}
Let $M,N,P$ be $\bW$-modules as above. Then, 
$\dim \overline{\B}_{\bW}(M\times N,P)=1$  iff
$M=B_\infty$. 
\end{lemma}

The proof  is similar to that of Lemma \ref{prop_002}.

\section{Computation of $\overline{\B}_{\bW}(M\times N,P)$ when
$\dim\cG_{\bW}(M\times N,P)=2$}

\subsection{The Theorem 3.2}

The following table provides a list of triples $(M,\,N,\,P)$  of $\bW$-modules of
the class  $\cS$ and some  elements
$\pi\in {\overline\B}_{\bW}(M\times N,P)$. 
For each entry $(M,\,N,\,P)$
we have  $\deg M=\deg N=\deg P= \{0,1\}$

\begin{table}[h]
\caption{\textbf{The non-zero $\overline{\B}_{\bW}(M\times N,P)$,
with $\deg M=\deg N=\deg P= \{0,1\}$}}
\label{table5}

\begin{center}
\begin{tabular}{|c|c|c||c|} \hline

 & $M\times N$or $N\times M$ &$P$& Elements of ${\overline \B}_{\bW}(M\times N,P)$ 
\\ \hline

1. &$\Omega^0_u\times \Omega^0_{v}$ & $\Omega^1_{u+v}$& 
$P^{0,0}_{u,v}\circ(d\times id)$ and $P^{0,0}_{u,v}\circ(id\times d)$    \\ \hline

2.& $\Omega^0_u\times \Omega^0_{-u}$ & $\Omega^0_{0}$&$P^{0,0}_{u,-u}$  \\ \hline

3.& $\Omega^0_u\times \Omega^1_{0}$ & $\Omega^1_{u}$&$P^{0,1}_{u,0}$  \\ \hline

4.& $\Omega^0_u\times \Omega^0_{-u}$ & $A_{\xi}$ & $d^\xi\circ P^{0,0}_{u,-u}$  \\
\hline

5.& $B_\xi\times \Omega^0_{u}$ & $\Omega^1_u$ & $P^{1,0}_{0,u}\circ(d_\xi\times id)$  \\
\hline

\end{tabular}
\end{center}
\vbox
{\small{To avoid repetitions, one can assume that $\xi\neq\infty$ in lines 4 or
5}}

\end{table}

From the table, it is clear that

(i) If $(M,\,N,\,P)$ appears in line 1 of the Table \ref{table5}, the two listed elements
of ${\overline \B}_{\bW}(M\times N,P)$ are linearly independent and therefore
${\overline \B}_{\bW}(M\times N,P)$ has dimension 2,

(ii) if $(M,\,N,\,P)$ appears in lines 2-5 of the Table \ref{table5}, the  listed element
of ${\overline \B}_{\bW}(M\times N,P)$ is not zero and therefore
${\overline \B}_{\bW}(M\times N,P)$ has dimension $\geq 1$.

Indeed, we have

\setcounter{thm}{2}
\begin{thm}\hskip-1.7mm{\bf 2} 
Let $M, N$ and $P$ be $\bW$-modules of the class 
$\cS$ with $\deg M=\deg N=\deg P= \{0,1\}$. Then

(i) if $(M,\,N,\,P)$ appears in line 1 of the Table \ref{table5},
we have $\dim {\overline \B}_{\bW}(M\times N,P)=2$,

(ii) if $(M,\,N,\,P)$ appears in lines 2-5 of the Table \ref{table5},
we have $\dim {\overline \B}_{\bW}(M\times N,P)=1$,

(iii) otherwise, we have ${\overline \B}_{\bW}(M\times N,P)=0$.
\end{thm}

In order to prove Theorem 3.2, note that the
case where $M$, $N$ and $P$ are irreducible is already treated in
Lemma \ref{lemma_simple}. Thus, up to
the $\fS_3$-symmetry, there are only three cases to consider:

\begin{enumerate}
\item $(M,N)=(A_{\xi_1}, A_{\xi_2})$ and $P=A_\xi$ or $P=B_\xi$,

\item $(M,N)=(B_{\xi_1}, B_{\xi_2})$ and $P=A_{\xi_3}$ or $P=B_{\xi_3}$,

\item $(M,N)=(\Omega_u^0, \Omega_{-u}^0)$ with $u \not\equiv 0 \mod \nz$
and $P=A_\xi$ or $P=B_\xi$.
\end{enumerate}

These cases will be treated in the next three subsections.

\subsection{The case $(M,N)=(A_{\xi_1}, A_{\xi_2})$}

\begin{lemma}\label{prop_AAA}
Any bilinear map in $\B_{\bW}(A_{\xi_1}\times A_{\xi_2}, P)$,
where $P$ is in the $AB$-family,  is degenerate.
\end{lemma}

\begin{proof} Any module $P$ in the $AB$-family admits an
almost-isomorphism to a module of the $A$-family. So we can 
 assume $P\simeq A_{\xi_3}$ for some $\xi_3$ in $\np^1$.  

Fix a basis $\{e_m^M\}$, 
$\{e_m^N\}$ and $\{e_m^P\}$ of $M=A_{\xi_1}$, $N=A_{\xi_2}$ and 
$P=A_{\xi_3}$, respectively, as in Section
\ref{sect_KS}. We have $L_{-1}.e_0^M\neq 0$ or $L_{1}.e_0^M\neq 0$.
Since both cases are similar, we case assume 
that $L_{-1}eu_0^M\neq 0$. Using that
$L_{-1}.e_1^N=L_{-1}.e_1^P= 0$, we obtain that
$\pi(L_{-1}.e_0^M,e_1^N)=L_{-1}.\pi(e_0^M,e_1^N)-
\pi(e_0^M,L_{-1}.e_1^N)=0$, hence we have

\noindent\centerline{$\pi(e_{-1}^M,e_1^N)=0$.}

Since $M_{\leq -1}$
(respectively $N_{\geq 1}$) is an
irreducible Verma $\fsl(2)$-module (respectively the restricted dual of
an irreducible Verma $\fsl(2)$-module), it follows that 
$M_{\leq -1}\otimes N_{\geq 1}$ is generated by $e_{-1}^M\otimes e_1^N$.
Hence we get $\pi(M_{\leq -1}\times  N_{\geq 1})=0$. Since 
$N_{\geq 1}$ is a $\bW_{\geq 0}$-submodule and since 
$\overline A$ is the  $\bW_{\geq 0}$-submodule generated by
 $M_{\leq -1}$, we have

\noindent\centerline{$\pi({\overline A}\times  N_{\geq 1})=0$,}

\noindent from which it follows that $\pi$ is degenerate. 
\end{proof}

\subsection{The case $(M,N)=(B_{\xi_1}, B_{\xi_2})$}

Recall that
$\d_\xi\circ P_{0,0}^{0,0}$ is a non-degenerate bilinear map
from $B_\infty\times B_\infty\rightarrow A_\xi$, for any $\xi\in\np^1$.
Let $\xi_1,\xi_2$ and $\xi_3\in\np^1$, and let $s$ be the number
of indices $i$ such that $\xi_i=\infty$. 

 By $\fS_3$ symmetry, 
${\overline \B}_{\bW}(B_{\xi_1}\times B_{\xi_2},A_{\xi_3})$ 
is not zero $s\geq 2$. More precisely, we have:

\begin{lemma} We have

(i) $\dim {\overline \B}_{\bW}(B_{\xi_1}\times B_{\xi_2},A_{\xi_3})=s-1$
if $s\geq 2$ 

(ii) $\dim {\overline \B}_{\bW}(B_{\xi_1}\times B_{\xi_2},B_{\xi_3})=1$
if $s=3$

(iii) $\dim {\overline \B}_{\bW}(B_{\xi_1}\times B_{\xi_2},A_{\xi_3})=
\dim {\overline \B}_{\bW}(B_{\xi_1}\times B_{\xi_2},B_{\xi_3})=0$
otherwise.

\end{lemma}

\begin{proof}
Let $(a,b), (c,d)$ and $(e,f)$ be projective coordinates of $\xi_1,\xi_2$ and $\xi_3 \in \np^1$, respectively, and fix a basis $\{e_m^M\}$, 
$\{e_m^N\}$ of $M=B_{a,b}$ and $N=B_{c,d}$, respectively, as in Section \ref{sect_KS}.  

First, we consider $\pi \in \B_{\bW}(B_{a,b} \times B_{c,d}, P)$ with $P=B_{e,f}$. 
Let $\{e_m^P\}$ be a basis of $B_{e,f}$ as in $\S$ \ref{sect_KS}. By calculating
$L_{-n}.\pi(e_n^M,e_0^N)$ and $L_{-n}.\pi(e_0^M,e_n^N)$, we see that $\xi_1=\xi_2=\xi_3 \in \np^1$. Hence, we may assume that $(a,b)=(c,d)=(e,f)$ without loss of generality. Similarly, by calculating 
$L_m.\pi(e_n^M,e_0^N)$ and $L_m.\pi(e_0^M,e_n^N)$, we see that there exists a constant $C \in \nc$ satisfying $\pi(e_m^M,e_0^N)=Ce_m^P=\pi(e_0^M,e_m^N)$ for any $m$.
It can be shown that such a $\bW$-equivariant map exists only if $a=0$ or $C=0$. In the former case, $\pi$ is a scalar multiple of $P_{0,0}^{0,0}$.  In the latter case, $\pi$ factors through $\overline{A}\times \overline{A}$ and one can apply a similar argument to the latter half of the proof of Lemma \ref{prop_AAA} to see that $\pi$ is degenerate. 

Second, we consider $\pi \in B_{\bW}(B_{a,b} \times B_{c,d}, P)$ with $P=A_{e,f}$. 
Let $\{e_m^P\}$ be a basis of $A_{e,f}$ as in $\S$ \ref{sect_KS}. 
By an argument by restriction, one sees that there exists constants $C_1, C_2 \in \nc$ such that $\pi(e_m^M, e_n^N)=(C_1m+C_2n)e_{m+n}^P$ for $m,n, m+n\neq 0$. Set $\pi(e_m^M,e_0^N)=a(m)e_m^P, \pi(e_0^M, e_m^N)=b(m)e_m^P$ and $\pi(e_m^M, e_{-m}^N)=c(m)e_0^P$. It is clear that $a(0)=b(0)=c(0)=0$.
By calculating $L_{-n}.\pi(e_m^M, e_n^N)$ with $m,n, m+n\neq 0$, one obtains
that $a(m)=-d^{-1}C_1m$ if $c=0$ and that $a(m)=C_1=0$ otherwise. Similarly, 
on obtains that $b(m)=-b^{-1}C_2m$ if $a=0$ and that $b(m)=C_2=0$ otherwise. 
Finally, by calculating $L_n.\pi(e_m^M, e_{-m}^N)$ with $n\neq \pm m$, on obtains that $c(m)=-(C_1-C_2)f^{-1}m$ if $e=0$ and that $c(m)=0$ and $C_1=C_2$ otherwise. Hence, it follows that $\dim \overline{\B}_{\bW}(B_{a,b} \times B_{c,d},A_{e,f})$ is equal to $0$ if $s\leq 1$ and is less than $s-1$ if $s\geq 2$. 
Now, for $s\geq 2$, the result follows from Table \ref{table5}.
\end{proof}

\subsection{The case $(M,N)=(\Omega_u^0, \Omega_{-u}^0)$ with 
$u \neq 0 \in \nc/\nz$}

The next lemma can be proved in a way similar to the proof of Lemma \ref{prop_002}.
\begin{lemma} Let $\xi\in\np^1$. We have 

(i)  $\dim {\overline \B}_{\bW}(\Omega_u^0\times\Omega_{-u}^0,A_{\infty})=2$
and $\dim {\overline \B}_{\bW}(\Omega_u^0\times\Omega_{-u}^0,A_{\xi})=1$
if $\xi\neq\infty$.

(ii)  $\dim {\overline \B}_{\bW}(\Omega_u^0\times\Omega_{-u}^0,B_{\infty})=1$
and $\dim {\overline \B}_{\bW}(\Omega_u^0\times\Omega_{-u}^0,B_{\xi})=0$
if $\xi\neq\infty$.

\end{lemma}

\section{Conclusion}\label{sect_conclusion}

From the classification, we can derive the following corollaries,
some of which are used in \cite{IM}.

\begin{cor}\label{cor1}
The primitive bilinear maps between modules of the class $\cS$ are the following:

(i) the Poisson products $P^{\delta_1,\delta_2}_{u,v}$, 

(ii) the Poisson brackets $B^{\delta_1,\delta_2}_{u,v}$ for
$\delta_1.\delta_2.(\delta_1+\delta_2)\neq 0$

(iii) the Lie brackets $B^{\delta_1,\delta_2}_{u,v}(\xi)$ for
$\delta_1.\delta_2.(\delta_1+\delta_2)= 0$ and $\xi\neq \infty$ if $\delta_1=\delta_2=0$,

(iv) $\Theta_\infty$, 

(iv) the Grozman operation $G_{u,v}$,

(vi) $\eta(\xi_1,\xi_2,\xi_3)$ for $\xi_1,\xi_2$ and $\xi_3$ are all distinct,
and their $\fS_3$-symmetric

(vii) the obvious map $P^M$, and their $\fS_3$-symmetric.
\end{cor}

This follows easily by closed examination. 

\begin{cor}
Let $M,\,N,\,P\in\cS^*$ such that 
$\B^0_\bW(M\times N,P)\neq 0$. Then the number of reducible modules among
$M$, $N$ and $P$ is $1$ or $3$.
\end{cor}

\begin{proof}
Let $\pi\in\B_{\bW}^0(M\times N,P)$ be non-zero. It follows from 
Theorem 1 that the three modules are reducible whenever
$\Supp\,\pi$ is not one line. Otherwise, we can assume that.
Assume $\Supp\,\pi=V$. Then $M$, which admits a trivial quotient is reducible.
Moreover, there is an almost-isomorphism $\phi:N\rightarrow P$, which
proves that $N$ and $P$ are simultaneously reducible or irreducible.
\end{proof}

\begin{cor}
Let $M,\,N,\,P\in\cS^*$. If
$\B_\bW(M\times N,P)\neq 0$, then we have

\noindent\centerline{
$(\deg\,M,\,\deg\,N,\,\deg\,P)\in\fz$.}
\end{cor}

\begin{proof}
If $\cG_\bW(M\times N,P)\neq 0$, then the corollary follows from Theorem \ref{thm2}.
Otherwise, we have $\B^0_\bW(M\times N,P)\neq 0$.  If 
$\{\deg\,M,\,\deg\,N,\,\deg\,P\}\subset\{0,1\}$, then 
$(\deg\,M,\,\deg\,N,\,\deg\,P)$ belongs to $\fz$. Otherwise,
$\overline{\Supp\,\pi}$ is one line for any  non-zero $\pi\in \B^0_\bW(M\times N,P)$.
By the $\fS_3$-symmetry, we can assume that  $\overline{\Supp\,\pi}=V$. In such a case,
$N$ is isomorphic to $P$ and we have
$\deg\,M\in\{0,1\}$ and $\deg\,N=\deg\,P\not\in\{0,1\}$. It follows that
$(\deg\,M,\,\deg\,N,\,\deg\,P)$ belongs to $\fz$ as well.
\end{proof}

Let $M,\,N,\,P\in\cS$. The triple $(M,\,N,\,P)$ is called {\it mixing}
if we have $\B^0_\bW(M\times N,P)\neq 0$ and $\overline{\B}_\bW(M\times N,P)\neq 0$.
Here is a table of example of mixing triples:

\begin{table}[h]
\caption{\textbf{Example of mixing triples $(M,N,P)$}}
\label{table_mixing}

\begin{center}
\begin{tabular}{|c|c|c||c|} \hline

 $M\times N$ or $N\times M$ &$P$& A non-degenerate $\pi_1$ 
&A degenerate $\pi_2$  
\\ \hline

$\Omega^0_u\times \Omega^1_{0}$ & $\Omega^1_{u}$& 
$P^{0,1}_{u,0}$ & $(f,\alpha)\mapsto (\Res\,\alpha)\d f $   \\ \hline

 $\Omega^0_u\times \Omega^0_{-u}$ & $\Omega^0_{0}$&$P^{0,0}_{u,-u}$ 
& $(f,g)\mapsto \Res\,f \d g$\\ \hline

\end{tabular}
\end{center}
\end{table}

\begin{cor}
Any mixing triple  $(M,\,N,\,P)$ appears in the Table \ref{table_mixing}
and in each case $\pi_1$ and $\pi_2$ form a basis of 
$\B_\bW(M\times N,P)$.
\end{cor}

The corollary follows immediately from Tables \ref{table2} and \ref{table4}.

\begin{cor}
For any triple  $(M,\,N,\,P)$ of indecomposable $\bW$-modules in the class $\cS$,
we have $\dim \B_{\bW}(M\times N,P)\leq 2$.
\end{cor}

This follows easily from Theorem 1, Theorem 2 and the previous corollary.
Note that the hypothesis that $M,N$ and $P$ are indecomposable is necessary.
For example we have $\dim \B_{\bW}(X\times X,X)=4$ if
$X=\overline{A}\oplus \nc$.

\appendix

\section{Complete expression for $\E_{\delta_1,\delta_2,\gamma}(x,y)$}

Recall that $\E_{\delta_1,\delta_2,\gamma}=\det\M$, where $\M$ is a $6\times 6$ matrix whose
entries are polynomials of degree $2$ in $x,y,\delta_1,\delta_2$ and $\gamma$. 
The appendix provides the explicit formula for $\E$.

It is quite obvious that 
$\E_{\delta_1,\delta_2,\gamma}(x,y)=\E_{\delta_1,\delta_2,\gamma}(-x-7,-y+7)$.
Thus, in order to provide more compact formulas, it is better to
list the polynomials $\tilde{q}_{i,j}$ defined by

\noindent \centerline{$\E_{\delta_1,\delta_2,\gamma}(x-7/2,y+7/2)=C(\delta_1,\delta_2,\gamma)
\sum_{i,j}\tilde{q}_{i,j}(\delta_1,\delta_2,\gamma)
x^iy^j$.}

\noindent The polynomials $\tilde{q}_{i,j}$ are calculated with MAPLE. 
It turns out that the only  non-zero polynomials are 
$\tilde{q}_{0,0}, \tilde{q}_{0,2},\tilde{q}_{1,1},\tilde{q}_{2,0},
\tilde{q}_{1,3},\tilde{q}_{2,2}$ and $\tilde{q}_{3,1}$.
It follows that $\E_{\delta_1,\delta_2,\gamma}$ is a polynomial of
degree four in $x$ and $y$ and therefore
$\tilde{q}_{1,3}=q_{1,3}$, $\tilde{q}_{3,1}=q_{3,1}$ and $\tilde{q}_{2,2}=q_{2,2}$,
where the polynomials $q_{1,3}$, $q_{3,1}$ and $q_{2,2}$ are given in
Section \ref{sect_poly-pij}. The other non-zero polynomials $\tilde q_{i,j}$ are given by
the following formulas:

\begin{align*}
16\tilde{q}_{0,0}=
&[(4\delta_1+1)(4\delta_2+1)\gamma(1-\gamma)+16(\delta_1^2+\delta_2^2-\delta_1\delta_2)\delta_1\delta_2 \\
&\phantom{(}+4(\delta_1^2+\delta_2^2-3\delta_1\delta_2(\delta_1+\delta_2))
+(13(\delta_1^2+\delta_2^2)-50\delta_1\delta_2)+11(\delta_1+\delta_2)+2] \\
\times
&[(4\delta_1+1)(4\delta_2+1)\gamma(1-\gamma)++16(\delta_1^2+\delta_2^2-\delta_1\delta_2)\delta_1\delta_2 \\
&\phantom{(}+4(\delta_1^2+\delta_2^2)(\delta_1+\delta_2)
-3(\delta_1^2+\delta_2^2+6\delta_1\delta_2)-7(\delta_1+\delta_2)+6],
\end{align*}

\vspace{-0.3 in}
\begin{align*}
-4\tilde{q}_{0,2}=
&(4\delta_1+1)^2\gamma^2(1-\gamma)^2 \\
+&2(4\delta_1+1)((4\delta_1+1)\delta_2^2-2(4\delta_1+1)(\delta_1+1)\delta_2+4\delta_1^3+5\delta_1^2+2\delta_1+4)\gamma(1-\gamma) \\
+&(4\delta_1+1)^2\delta_2^4-4(4\delta_1+1)^2(\delta_1+1)\delta_2^3
+(32\delta_1^4+112\delta_1^3+142\delta_1^2+52\delta_1-13)\delta_2^2 \\
-
&(64\delta_1^5+32\delta_1^4-92\delta_1^3+68\delta_1^2+82\delta_1-4)\delta_2 \\
-
&(4\delta_1+1)(\delta_1-1)(\delta_1+1)(\delta_1+2)(4\delta_1^2+\delta_1-6),
\end{align*}
\vspace{-0.3 in}

\vspace{-0.3 in}
\begin{align*}
-4\tilde{q}_{2,0}=
&(4\delta_2+1)^2\gamma^2(1-\gamma)^2 \\
+&2(4\delta_2+1)((4\delta_2+1)\delta_1^2-2(4\delta_2+1)(\delta_2+1)\delta_1+4\delta_2^3+5\delta_2^2+2\delta_2+4)\gamma(1-\gamma) \\
+&(4\delta_2+1)^2\delta_1^4-4(4\delta_2+1)^2(\delta_2+1)\delta_1^3
+(32\delta_2^4+112\delta_2^3+142\delta_2^2+52\delta_2-13)\delta_1^2 \\
-
&(64\delta_2^5+32\delta_2^4-92\delta_2^3+68\delta_2^2+82\delta_2-4)\delta_1 \\
-
&(4\delta_2+1)(\delta_2-1)(\delta_2+1)(\delta_2+2)(4\delta_2^2+\delta_2-6),
\end{align*}

\begin{align*}
\frac{1}{2}\tilde{q}_{1,1}=
&(28\delta_1^2\delta_2^2-4(\delta_1+\delta_2)\delta_1\delta_2+(\delta_1^2+\delta_2^2)-20\delta_1\delta_2-(\delta_1+\delta_2))\gamma(1-\gamma) \\
&+4(7(\delta_1^2+\delta_2^2)-10\delta_1\delta_2)(\delta_1\delta_2)^2
-4(\delta_1^3+\delta_2^3)\delta_1\delta_2 \\
&+(\delta_1^4+\delta_2^4-12(\delta_1^2+\delta_2^2)\delta_1\delta_2-6(\delta_1\delta_2)^2)\\
&+2(\delta_1^2+\delta_1\delta_2+\delta_2^2)(\delta_1+\delta_2)
-(\delta_1^2+\delta_2^2-14\delta_1\delta_2)-2(\delta_1+\delta_2).
\end{align*}

\newpage
{\it Authors addresses:}

Universit\'e Claude Bernard Lyon 1,

Institut Camille Jordan, UMR 5028 du CNRS

43, bd du 11 novembre 1918

69622 Villeurbanne Cedex

FRANCE

{\it Email addresses:}

iohara@math.univ-lyon1.fr

mathieu@math.univ-lyon1.fr

\end{document}